\documentclass[11pt]{amsart}
\usepackage{amsmath,amssymb,amsthm,times}
\usepackage{microtype}
\usepackage[colorlinks=true]{hyperref}

\numberwithin{equation}{section}
\theoremstyle{plain}
\newtheorem{theorem}{Theorem}[section]
\newtheorem{lemma}[theorem]{Lemma}
\newtheorem{proposition}[theorem]{Proposition}
\newtheorem{corollary}[theorem]{Corollary}
\theoremstyle{definition}

\theoremstyle{remark}
\newtheorem{remark}[theorem]{Remark}

\allowdisplaybreaks[4]
\newcommand{\dd}{\,\mathrm{d}}
\newcommand{\Div}{\operatorname{div}}
\newcommand{\curl}{\operatorname{curl}}
\newcommand{\supp}{\operatorname{supp}}
\newcommand{\Acal}{\mathcal{A}}

\begin{document}
\title[Finite Potential Energy for Planar Ginzburg--Landau Solutions]
{Finite Potential Energy for Entire Solutions of the Planar Ginzburg--Landau Equation}

\author[H.G. Chen]{Hong-Ge Chen}
\address{Hong-Ge Chen, School of Mathematics and Statistics, Key Laboratory of Nonlinear Analysis \& Applications (Ministry of Education),
Central China Normal University, Wuhan, China}
\email{hongge\_chen@whu.edu.cn}

\author[J.C. Wei]{Juncheng Wei}
\address{Juncheng Wei, Department of Mathematics, Chinese University of Hong Kong, Shatin, New Territories, Hong Kong}
\email{wei@math.cuhk.edu.hk}

\author[H.C. Yan]{Haicheng Yan}
\address{Haicheng Yan, School of Mathematics and Statistics, Wuhan University, Wuhan, China}
\email{yanhaicheng@whu.edu.cn}

\author[W. Yang]{Wen Yang}
\address{Wen Yang, Department of Mathematics, Faculty of Science and Technology, University of Macau, Taipa, Macau, China}
\email{wenyang@um.edu.mo}

\begin{abstract}
We prove that every smooth entire solution $ u\colon\mathbb{R}^2\to\mathbb{R}^2  $ of the Ginzburg--Landau equation $  -\Delta u=u(1-|u|^2)  $ with $  |u(x)|\to1  $ as $  |x|\to\infty  $ has finite potential energy, i.e.,
\begin{equation*}
\int_{\mathbb{R}^2}(1-|u|^2)^2\dd x<+\infty,
\end{equation*}
thereby resolving Brezis' Open Problem 2.5 in \cite{BrezisProblems}. The main difficulty stems from the possible presence of a curl-free mode that carries nonzero circulation and decays only like $  |x|^{-1}  $; such a mode lies outside $  L^2  $ and does not admit a single-valued potential. By minimizing over $  L^2  $ gradient corrections, we construct a comparison field that solves the homogeneous equation and inherits the same circulation. The Kelvin inversion, combined with the De Giorgi--Nash--Moser theory for quasilinear elliptic equations, then produces the optimal decay $  O(|x|^{-1})  $. For a Ginzburg--Landau solution, the Bernstein estimate and the coercivity of the Jacobi form produce an $  L^2  $ forcing term in the exterior phase equation. The resulting $  L^4  $ bound on the phase field implies $1-|u|^2\in L^2(\mathbb{R}^2)$, and therefore the potential energy is finite.
\end{abstract}


\maketitle

\section{Introduction}
The Ginzburg--Landau theory of superconductivity, introduced by Ginzburg and Landau in 1950, describes superconducting states through a complex order parameter whose zeros identify vortex cores.  In the nonmagnetic two-dimensional reduction, the energy of a map $v\colon\mathcal O\to\mathbb C$ on a domain $\mathcal O\subset\mathbb R^2$ is 
\[ \mathcal E_\varepsilon(v;\mathcal O)
 =\frac12\int_{\mathcal O}|\nabla v|^2\dd x
 +\frac1{4\varepsilon^2}
  \int_{\mathcal O}(1-|v|^2)^2\dd x,
\]
where $\varepsilon>0$ is the small coherence-length parameter.  For minimizers on a smooth bounded star-shaped planar domain with smooth $S^1$-valued Dirichlet data of degree $d>0$, Bethuel, Brezis, and H\'{e}lein \cite{BBH} proved that, along subsequences as $\varepsilon\to0$, the energy concentrates near $d$ distinct degree-one vortices and that their limiting configuration minimizes the associated renormalized energy. For $d<0$, the analogous conclusion follows by complex conjugation and involves $|d|$ vortices of degree $-1$. Pacard and Rivi\`{e}re \cite{PacardRiviere}, Del Pino, Kowalczyk and Musso \cite{DKM} (see also Del Pino, Juneman and Musso \cite{DJM})  subsequently developed the linear and nonlinear theory of such vortices, while Sandier and Serfaty \cite{SandierSerfaty} treated the corresponding magnetic model.

After the natural rescaling that eliminates $\varepsilon$, one is led to the entire equation
\begin{equation}\label{eq:GL-intro}
-\Delta u=u(1-|u|^2)
\quad\text{in }\mathbb{R}^2,
\qquad u\colon\mathbb R^2\to\mathbb R^2.
\end{equation}
We are interested in smooth solutions that satisfy 
\begin{equation}\label{eq:modulus-limit-intro}
|u(x)|\to1\qquad\text{as }|x|\to\infty.
\end{equation}
The condition \eqref{eq:modulus-limit-intro} guarantees that the topological degree at infinity, defined for any sufficiently large $R$ by
\[
\deg(u,\infty)
:=\deg\left(\frac{u}{|u|}\Big|_{\partial B_R(0)}\right)
\in\mathbb Z
\]
is well defined and independent of the choice of $R$. However, it does not provide any quantitative decay estimate for the modulus defect $1-|u|^2$, nor any a priori bound on the potential energy
\[
\mathcal P(u)
:=\int_{\mathbb R^2}(1-|u|^2)^2\,\dd x.
\]
The potential energy is the natural global quantity associated with \eqref{eq:GL-intro}: for every $q\in \mathbb{Z}\setminus \{0\}$, the classical equivariant vortex $V_q(re^{i\theta})=f_q(r)e^{iq\theta}$ has finite potential energy, whereas its Dirichlet energy diverges logarithmically on large disks.  Thus, within this class, the finiteness of $\mathcal P$ is strictly weaker than the finiteness of the full Ginzburg--Landau energy.

Brezis, Merle, and Rivi\`ere \cite[Theorem~1 and Remarks~1.1--1.2]{BMR} proved that every smooth entire solution satisfies the quantization alternative
\[
  \mathcal P(u)
  \in
  \bigl\{2\pi d^2:d\in\mathbb N\cup\{0\}\bigr\}
  \cup\{+\infty\}.
\]
When $\mathcal P(u) <\infty$, the corresponding integer is $d=|\deg(u,\infty)|$, and $d=0$ implies that $u$ is constant of unit
modulus.  Furthermore,  Shafrir \cite[Theorem~1]{Shafrir} obtained precise far-field asymptotics
\[
  1-|u(x)|^2
  =\frac{d^2}{|x|^2}+o(|x|^{-2})
  \qquad\text{as }|x|\to\infty,
\]
and
Mironescu \cite[Theorem~1]{Mironescu} classified all solutions with $\deg(u,\infty)=\pm1$.  On the variational side, Sandier
\cite[Theorem~1.1]{Sandier} proved that every local minimizer has finite potential energy without assuming
\eqref{eq:modulus-limit-intro}; this result is based instead on the additional hypothesis of local minimality.

Related classification, Liouville, and uniqueness results, as well as a priori estimates, have been obtained under finite-potential-energy, variational, or
bounded-domain hypotheses.  Farina \cite[Theorem~1.1]{FarinaClass} proved that, in dimensions $N=3$ and $4$, every complex-valued entire local minimizer satisfying $|u(x)|\to1$ as $|x|\to\infty$ is constant of unit modulus.  For arbitrary entire solutions $u\colon\mathbb R^N\to\mathbb R^M$, he proved the corresponding finite-potential-energy Liouville property when $N\geq4$ and $M\geq1$ \cite[Theorem~1.1]{FarinaTwo}, and when $N=3$ and $M=2$ \cite[Theorem~1.1]{FarinaLiouville}.

For bounded-domain variational problems, Ignat, Nguyen, Slastikov, and Zarnescu proved that, for every $\varepsilon>0$ and $N\geq7$, the radial degree-one vortex is the unique global minimizer in the unit ball with vortex boundary data, under the convexity and positivity assumptions on the potential specified in \cite[Theorem~1]{IgnatNguyenSlastikovZarnescuCR}. Under strict convexity of the potential and a one-sided condition on the boundary data, they subsequently characterized uniqueness and its failure \cite[Theorem~1.4]{IgnatNguyenSlastikovZarnescuUniq}.
Ignat, Nahon, and Nguyen proved that, for $4\leq N\leq6$, the radial vortex is the unique minimizer among gradient fields, under the convexity
assumptions imposed in \cite[Theorem~1]{IgnatNahonNguyen}. More recently, assuming $W(0)=0$, $W\geq0$, and continuous unit-length Dirichlet data, Ignat, Nguyen, Slastikov, and Zarnescu proved the maximum principle $|u_\varepsilon|\leq1$ for bounded critical points on bounded Lipschitz domains \cite[Theorems~1.1 and~1.3]{Ignat}. Under their additional structural assumptions on $W$, energy convergence to a continuous stationary harmonic map, and regular boundary data, they also obtained uniform Laplacian and $C^{1,\beta}$ estimates. These results rely on finite potential energy, variational assumptions, dimension-dependent structure, or bounded-domain hypotheses.  None of them settles whether an arbitrary entire planar solution satisfying only \eqref{eq:modulus-limit-intro} must have finite potential energy.

Returning to the entire planar setting, Brezis, Merle, and Rivi\`{e}re asked in \cite[Problem~2]{BMR} whether the equation together
with the condition at infinity forces $\mathcal P(u)<\infty$. Brezis later restated the question in \cite[Open Problem~2]{brezis} and \cite[Open Problem~2.5]{BrezisProblems}.  In the notation used here, the natural and long-standing question is:
\begin{equation*}
\mbox{\bf Assume that $u$ satisfies \eqref{eq:GL-intro} and \eqref{eq:modulus-limit-intro}. Does it follow that $\mathcal P(u)<\infty$?}   
\end{equation*}
The present paper answers this question affirmatively.

\begin{theorem}
\label{thm:main}
Let $u\in C^\infty(\mathbb{R}^2;\mathbb{R}^2)$ satisfy
\begin{equation}\label{eq:GL}
 -\Delta u=u(1-|u|^2)
 \quad\text{in }\mathbb{R}^2
\end{equation}
and
\begin{equation}\label{eq:modulus-limit}
 |u(x)|\to 1
 \quad\text{as }|x|\to\infty.
\end{equation}
Then
\begin{equation*}\label{eq:main-conclusion}
 \int_{\mathbb{R}^2}(1-|u|^2)^2\dd x<\infty.
\end{equation*}
\end{theorem}

Combining Theorem~\ref{thm:main} with the quantization result in \cite{BMR}, we obtain
\[
  \mathcal P(u)=2\pi\,\deg(u,\infty)^2.
\]
Within this class of solutions, Theorem~\ref{thm:main} also removes the assumption $\mathcal P(u)<\infty$ from the degree-zero rigidity
conclusion of \cite[Remark~1.2]{BMR} and from Mironescu's classification of entire solutions of degrees $\pm1$
\cite[Theorem~1]{Mironescu}.  This conclusion should be distinguished from Sandier's theorem on locally minimizing solutions
\cite[Theorem~1.1]{Sandier}: the present argument relies only on the equation and the condition at infinity and requires neither local minimality nor stability.  The classification problem for degrees with absolute value at least two remains open.

The proof of Theorem~\ref{thm:main} relies on a sharp exterior-phase estimate, stated as Theorem~\ref{thm:exterior-phase} below. This estimate is of independent interest; it is formulated on exterior domains and permits an $  L^2  $ forcing term in an autonomous phase equation.

\begin{theorem}
\label{thm:exterior-phase}
Let $r_0>0$ and
\[
 E:=\mathbb{R}^2\setminus\overline{B_{r_0}(0)}
   =\{x\in\mathbb{R}^2:|x|>r_0\}.
\]
Assume that $k\in C^\infty(\mathcal U;\mathbb{R}^2)$ for some open
neighborhood $\mathcal U$ of $\overline E$.  Let
$F\in L^2(E;\mathbb{R}^2)$, and suppose that
\begin{equation}\label{eq:phase-lemma-assumptions}
 \curl k=0,
 \qquad
 \Div\bigl((1-|k|^2)k-F\bigr)=0
 \quad\text{in }\mathcal D'(E),
\end{equation}
and that
\begin{equation}\label{eq:k-uniform-zero}
 \lim_{R\to\infty}
 \sup_{\substack{x\in E\\ |x|>R}}|k(x)|=0.
\end{equation}
Then
\begin{equation}\label{eq:phase-prop-conclusion}
 k\in L^{\gamma}(E;\mathbb{R}^2)
 \quad\forall\,\gamma>2.
\end{equation}
\end{theorem}

\noindent {\bf Remark}: The restriction $\gamma>2$ is sharp; the circulation example and the endpoint failure are provided in Remark~\ref{rem:phase-sharp}.
\medskip

The major difficulty addressed by Theorem~\ref{thm:exterior-phase} is the possible presence of a mode of order $|x|^{-1}$ that carries nonzero circulation. This mode
lies outside $L^2$ and does not admit a single-valued potential. To overcome this
difficulty, we employ a global variational comparison inspired by classical
work on exterior flows and nonlinear Hodge theory. Shiffman introduced a
globally regular convex extension and renormalized variational problems in
exterior domains that incorporate a logarithmic singularity carrying
prescribed circulation \cite[Sections~3--4, pp.~612--616, and Section~16,
pp.~643--647]{Shiffman}. Sibner and Sibner developed a nonlinear Hodge theory
with prescribed periods in an $L^2$ setting, including the regular case on
noncompact manifolds \cite[Sections~1.3 and 3.2]{SibnerSibner1970}. Dong and
Ou formulated a disturbance-energy approach to exterior potential flow in
dimensions $n\geq 3$ and highlighted the special functional difficulties that
arise in dimension two \cite[pp.~358--359 and p.~365]{DongOu1993}. Building on
these ideas, we construct a homogeneous comparator $\ell$ that inherits the
circulation of $k$ and satisfies
\begin{equation*}
\|k-\ell\|_{L^2(E)}\leq C\|F\|_{L^2(E)}.
\end{equation*}
The construction allows the circulation-carrying fields $k$ and $\ell$
themselves to lie outside $L^2(E)$.

The comparison is performed as follows. Since $A(p)=(1-|p|^2)p$ is not monotone for large $|p|$, using Lemma~\ref{lem:convex-extension}, we replace $A$ outside a small ball in the target space by $\widetilde A=\nabla\widetilde W$, where
$\lambda I\leq D\widetilde A\leq I$ for some $\lambda>0$. On the exterior
domain $E$ one sets
\begin{equation*}
\mathcal G(E):=
\overline{\{\nabla\phi:\phi\in C_c^\infty(E)\}}^{L^2(E;\mathbb R^2)}
\end{equation*}
and minimizes the renormalized functional
\begin{equation*}
\mathcal J(w)=\int_E\bigl(\mathcal R_k(w)-F\cdot w\bigr)\,dx
\end{equation*}
over $\mathcal G(E)$; see Proposition~\ref{prop:global-variational-comparison}.
Here, $\mathcal R_k(w)$ is the Taylor remainder of $\widetilde W$ along the
segment from $k$ to $k-w$. The Hessian bounds imply that $\mathcal J$ is strictly convex and coercive. Since  $\mathcal J$ is also weakly lower semicontinuous on the closed subspace $\mathcal G(E)$, the direct method gives a unique minimizer. The corresponding Euler--Lagrange equation is
\begin{equation*}
\int_E\bigl(\widetilde A(k)-\widetilde A(k-w)-F\bigr)\cdot h\,dx=0
\qquad\text{for every }h\in\mathcal G(E).
\end{equation*}
With $\ell:=k-w$, Proposition~\ref{prop:global-variational-comparison} gives
$\operatorname{div}\widetilde A(\ell)=0$, $\operatorname{curl}\ell=0$, and
$\|k-\ell\|_{L^2}\leq C\|F\|_{L^2}$. The uniform decay of $\ell$ follows by
combining a local estimate with the vanishing $L^2$ tail of $w$ and the uniform
decay of $k$. Moreover, Lemma~\ref{lem:weak-circulation} shows that the curl-free  $L^2$ correction $w$ has zero circulation along every circle in $E$, so $\ell$ and $k$ share the same circulation.

The difference-quotient argument then shows that each component of $\ell$
satisfies the uniformly elliptic equation
\begin{equation*}
\operatorname{div}\bigl(D\widetilde A(\ell)\nabla\ell_m\bigr)=0.
\end{equation*}
The Kelvin inversion $x\mapsto x/|x|^2$ converts the exterior equation into a
divergence-form equation on a punctured disk. The matrix of reflected coefficients 
remains symmetric and uniformly elliptic. Since $\ell$ tends uniformly to zero,
the inverted field $v$ is continuously extended by $v(0)=0$.
Lemma~\ref{lem:planar-puncture-removal}, proved with logarithmic cutoffs, gives
$v\in W^{1,2}_{\mathrm{loc}}$ across the origin and the weak equation on the
full disk. The De Giorgi--Nash--Moser theory first yields $v(y)=O(|y|^\beta)$ for
some $\beta\in(0,1)$. Because $\widetilde A=A$ near the origin and
\begin{equation*}
DA(p)-I=-|p|^2 I-2p\otimes p,
\end{equation*}
the inverted coefficient matrix extends as a $C^{0,\beta}$ field at the origin.
Interior Schauder estimates then give $v\in C^{1,\beta}$ locally and hence
$|v(y)|\leq C|y|$. Therefore, Lemma~\ref{lem:critical-decay} produces the
critical exterior decay
\begin{equation*}
|\ell(x)|\leq\frac{C}{|x|}
\end{equation*}
for all sufficiently large $|x|$. In particular, $\ell\in L^\gamma(E)$ for every
$\gamma>2$, which completes the proof of Theorem~\ref{thm:exterior-phase}.
The same strategy was previously used by two of the authors to obtain the decay rate of the KP-I
lump solution; see \cite{lwxy}.

Having established this general exterior decay result, we now apply it to the
Ginzburg--Landau equation. The Bernstein estimate
\cite[Theorem~3.5]{Smyrnelis} gives the following:
\begin{equation*}
|\nabla u|^2\leq 1-|u|^2\qquad\text{in }\mathbb R^2.
\end{equation*}
In a sufficiently large exterior domain we write $u=\rho n$,  where $\rho\geq\frac12$ and $n$ is $S^1$-valued, and let $k$ be the globally defined closed one-form associated with $n$. Applying the coercivity estimate for the Jacobi form to cut-off translational fields
yields 
\begin{equation*}
\nabla\rho,\quad D^2\rho,\quad\nabla k\in L^2.
\end{equation*}
Consequently,
\begin{equation*}
\sigma:=(1-|u|^2)-|k|^2=-\frac{\Delta\rho}{\rho}
\end{equation*}
belongs to $L^2$. The phase equation then takes the autonomous form required by
Theorem~\ref{thm:exterior-phase}, with forcing $F=\sigma k\in L^2$. Moreover,
$k(x)\to 0$ uniformly as $|x|\to\infty$. The exterior theorem therefore gives
$k\in L^4$, since $1-|u|^2=|k|^2+\sigma$ and $\sigma\in L^2(\Omega)$, it follows that $1-|u|^2\in L^2(\Omega)$. Hence, the potential energy is finite, which is precisely the assertion of
Theorem~\ref{thm:main}.

The paper is organized as follows. Section~\ref{sec:exterior-phase-theorem}
establishes the exterior phase theorem by constructing the global comparator
and deriving its critical decay via Kelvin inversion and elliptic regularity.
Section~\ref{sec:ginzburg-landau} recalls and applies the Bernstein estimate,
derives the Jacobi coercivity estimate, constructs the exterior phase, and
completes the proof of Theorem~\ref{thm:main}.

\bigskip
\noindent
\textbf{Notation}: Throughout the paper, $C> 0$ denotes a generic constant that may change from line to line. Moreover, $B_r(x)$ denotes the open ball in $\mathbb R^2$ centered at the point $x$ with radius $r$.

\section{An exterior phase theorem}
\label{sec:exterior-phase-theorem}

We prove Theorem~\ref{thm:exterior-phase} by constructing the homogeneous
 comparator in Proposition~\ref{prop:global-variational-comparison} and establishing its
$|x|^{-1}$ decay in Lemma~\ref{lem:critical-decay}.

For $x\in\mathbb{R}^2\setminus\{0\}$, let
$\tau(x):=(-x_2,x_1)/|x|$.
We write
\begin{equation}\label{eq:autonomous-flux-potential}
 A(p):=(1-|p|^2)p,
 \qquad
 W(p):=\frac12|p|^2-\frac14|p|^4,
 \qquad \forall\,p\in\mathbb{R}^2,
\end{equation}
so that $A=\nabla W$.  The radial and tangential eigenvalues of
$D^2W(p)$ are $1-3|p|^2$ and $1-|p|^2$, respectively. Thus, $W$ is
strictly convex on $\{p\in\mathbb{R}^2:|p|<1/\sqrt3\}$, but is not
convex on $\mathbb{R}^2$.

\begin{lemma}
\label{lem:convex-extension}
Let $0<\delta<\frac14$. Then there exist
$\widetilde W\in C^\infty(\mathbb{R}^2)$,
$\widetilde A:=\nabla\widetilde W$, and
$\lambda:=1-12\delta^2>0$ such that
\[
 \widetilde W(p)=W(p),
 \qquad
 \widetilde A(p)=A(p)
 \quad\forall\,p\in\mathbb{R}^2\text{ with }|p|\leq\delta,
\]
and
\begin{equation}\label{eq:convex-extension-hessian-statement}
 \lambda|\xi|^2
 \leq D^2\widetilde W(p)\xi\cdot\xi
 \leq|\xi|^2
 \qquad\forall\,p,\xi\in\mathbb{R}^2.
\end{equation}
\end{lemma}

\begin{proof}
Choose $\vartheta\in C^\infty([0,\infty))$ such that
$0\leq\vartheta\leq1$ and
\[
 \vartheta(s)=1\quad\forall\,s\in[0,\delta],
 \qquad
 \vartheta(s)=0\quad\forall\,s\in[2\delta,\infty).
\]
Set
\[
 \mathfrak b(s):=1-3s^2\vartheta(s),
 \qquad
 \mathfrak a(s):=\int_0^s\mathfrak b(t)\dd t.
\]
On the support of $\vartheta$ one has $s\leq2\delta$, and hence $\lambda\leq\mathfrak b(s)\leq1$ for all $s\in[0,\infty)$. Define
\[
 \widetilde W(p):=\int_0^{|p|}\mathfrak a(s)\dd s.
\]
For $0\leq s\leq\delta$, we have 
$\mathfrak a(s)=s-s^3$.  Therefore, for every
$p\in\mathbb{R}^2$ with $|p|\leq\delta$, we obtain 
\[
 \widetilde W(p)
 =\int_0^{|p|}(s-s^3)\dd s
 =\frac12|p|^2-\frac14|p|^4
 =W(p).
\]
In particular, $\widetilde W=W$ on
$B_\delta(0)$ and is smooth at $p=0$.

For $p\neq0$, the radial and tangential eigenvalues of
$D^2\widetilde W(p)$ are
\[
 \mathfrak b(|p|)
 \quad\text{and}\quad
 \frac{\mathfrak a(|p|)}{|p|}
 =\frac1{|p|}\int_0^{|p|}\mathfrak b(s)\dd s,
\]
respectively.  Both lie in $[\lambda,1]$, and the same bounds hold at
the origin by continuity.  This proves
\eqref{eq:convex-extension-hessian-statement}.  Integrating the Hessian
bounds along line segments also gives
\begin{equation}\label{eq:extension-lipschitz-monotone}
 \begin{aligned}
 |\widetilde A(p)-\widetilde A(q)|
 &\leq |p-q|,\\
 \bigl(\widetilde A(p)-\widetilde A(q)\bigr)\cdot(p-q)
 &\geq\lambda|p-q|^2
 \end{aligned}
 \qquad\forall\,p,q\in\mathbb{R}^2.
\end{equation}
\end{proof}

Next we present two interior estimates used below.  For $N\in\mathbb{N}^{+}$,
let $U\subset\mathbb{R}^2$ be an open subset, let
$f\colon U\to\mathbb{R}^N$, and let $\alpha\in(0,1)$.  We use the
H\"older seminorm
\[
 [f]_{C^{0,\alpha}(U;\mathbb{R}^N)}
 :=\sup_{\substack{x,y\in U\\x\neq y}}
 \frac{|f(x)-f(y)|}{|x-y|^\alpha}.
\]
For scalar-valued maps $  f  $, the target space is omitted and the seminorm is denoted simply by $  [f]_{C^{0,\alpha}(U)}  $. For a matrix $  M=(M_{ij})\in\mathbb{R}^{2\times 2}  $, we write $  |M| := \bigl(\sum_{i,j=1}^2 |M_{ij}|^2\bigr)^{1/2}  $ for its Frobenius norm. The same definition of the Hölder seminorm is used for matrix-valued maps, with the norm understood in the Frobenius sense. We now recall the following classical interior elliptic regularity estimates used in this section.

\begin{lemma}
\label{lem:interior-elliptic-estimates}
Let $x_0\in\mathbb{R}^2$ and $r>0$, and let $0<\lambda<\Lambda<+\infty$, and let
$G\in L^\infty(B_{2r}(x_0);\mathbb{R}^{2\times2})$ be symmetric and
satisfy
\[
 \lambda|\xi|^2\leq G(x)\xi\cdot\xi\leq\Lambda|\xi|^2
 \quad\text{for a.e. }x\in B_{2r}(x_0)
 \quad\forall\,\xi\in\mathbb{R}^2.
\]
If $q\in W^{1,2}(B_{2r}(x_0))$ and
\[
 \Div(G\nabla q)=0
 \quad\text{in }\mathcal D'(B_{2r}(x_0)),
\]
then there are $\alpha=\alpha(\lambda,\Lambda)\in(0,1)$ and
$C=C(\lambda,\Lambda)$ such that $q\in C^{0,\alpha}(B_r(x_0))$ and
\begin{equation}\label{eq:interior-DGN-estimate}
 \|q\|_{L^\infty(B_r(x_0))}
 +r^\alpha[q]_{C^{0,\alpha}(B_r(x_0))}
 \leq\frac{C}{r}\|q\|_{L^2(B_{2r}(x_0))}.
\end{equation}
In addition, if 
$G\in C^{0,\alpha}(B_{2r}(x_0);\mathbb{R}^{2\times2})$, then
$q\in C^{1,\alpha}(B_r(x_0))$ and
\begin{equation}\label{eq:interior-Schauder-estimate}
 r\|\nabla q\|_{L^\infty(B_r(x_0);\mathbb{R}^2)}
 +r^{1+\alpha}[\nabla q]_{C^{0,\alpha}
 (B_r(x_0);\mathbb{R}^2)}
 \leq C \|q\|_{L^\infty(B_{2r}(x_0))},
\end{equation}
where $C>0$ depends only on  $\lambda$,
$\Lambda$, $\alpha$, and 
$ r^\alpha[G]_{C^{0,\alpha}
 (B_{2r}(x_0);\mathbb{R}^{2\times2})}$. 
\end{lemma}

\begin{proof}
Although the equation is vectorial in appearance, the asserted regularity follows componentwise from the classical De Giorgi--Nash--Moser and Schauder theories for scalar uniformly elliptic equations in divergence form. We refer the reader to \cite[Chapter 8]{GT} for the detailed arguments.
\end{proof}
 
Next, we state a circulation lemma for weakly curl-free fields.

\begin{lemma}
\label{lem:weak-circulation}
Let $r>0$, set $V:=\{x\in\mathbb{R}^2:|x|>r\}$, and suppose that
\[
 w\in W^{1,2}_{\mathrm{loc}}(V;\mathbb{R}^2)
   \cap L^2(V;\mathbb{R}^2),
 \qquad
 \curl w=0\quad\text{in }\mathcal D'(V).
\]
For $s>r$, let
$\operatorname{Tr}_s w\in H^{1/2}(\partial B_s(0);\mathbb{R}^2)$
denote the Sobolev trace of $w$ on $\partial B_s(0)$.  Then
\[
 \int_{\partial B_s(0)}
 (\operatorname{Tr}_s w)\cdot\tau\,\dd\mathcal H^1=0\footnote{Here and below, $\mathcal H^1$ denotes one-dimensional Hausdorff
measure (arc length on circles), and $\mathcal L^2$ denotes
two-dimensional Lebesgue measure}
 \qquad\forall\,s>r,
\]
where $\tau(x)=(-x_2,x_1)/|x|$ is the positively oriented unit tangent to $\partial B_s(0)$.
\end{lemma}

\begin{proof}
For $s>r$, set
\[
 c(s):=\int_{\partial B_s(0)}
 (\operatorname{Tr}_s w)\cdot\tau\,\dd\mathcal H^1.
\]
Fix $r<s_1<s_2$ and set
$ \mathcal A_{12}:=B_{s_2}(0)\setminus\overline{B_{s_1}(0)}
$. Since $\overline{\mathcal A_{12}}\subset V$, we have 
$w\in W^{1,2}(\mathcal A_{12};\mathbb{R}^2)$.  The annulus
$\mathcal A_{12}$ is a bounded Lipschitz domain.  Choose a sequence $w_j\in C^\infty(\overline{\mathcal A_{12}};\mathbb{R}^2)$ such that
\[
 \begin{gathered}
  w_j\to w
  \quad\text{strongly in }
  W^{1,2}(\mathcal A_{12};\mathbb{R}^2).
 \end{gathered}
\]
By the continuity of the trace operator, we have 
\[
 \operatorname{Tr}w_j\to\operatorname{Tr}w
 \quad\text{strongly in }
 L^2(\partial\mathcal A_{12};\mathbb{R}^2).
\]
Applying the classical Stokes formula to $w_j$ and then passing to the
limit, we obtain 
\[
 \begin{aligned}
 c(s_2)-c(s_1)
 &=\lim_{j\to\infty}
 \left(
 \int_{\partial B_{s_2}(0)}w_j\cdot\tau\,\dd\mathcal H^1
 -\int_{\partial B_{s_1}(0)}w_j\cdot\tau\,\dd\mathcal H^1
 \right)\\
 &=\lim_{j\to\infty}
 \int_{\mathcal A_{12}}\curl w_j\,\dd x
 =\int_{\mathcal A_{12}}\curl w\,\dd x
 =0.
 \end{aligned}
\]
The last equality holds because $w\in W^{1,2}(\mathcal A_{12};\mathbb{R}^2)$ and $\curl w=0$ in the sense of distributions. Together, these imply that $\curl w = 0$ almost everywhere in $\mathcal A_{12}$. 
 Thus $c(s)=c$ for all $s>r$.

By Sobolev slicing in polar coordinates, for almost every $s>r$ the
trace $\operatorname{Tr}_s w$ coincides $\mathcal H^1$-almost everywhere
with the polar $L^2$-slice of $w$.  For such $s$, by Cauchy--Schwarz on
$\partial B_s(0)$, we infer 
\[
 \begin{aligned}
 |c|^2
 &=\left|
 \int_{\partial B_s(0)}w\cdot\tau\,\dd\mathcal H^1
 \right|^2\leq\mathcal H^1(\partial B_s(0))
 \int_{\partial B_s(0)}|w|^2\,\dd\mathcal H^1\\
 &=2\pi s
 \int_{\partial B_s(0)}|w|^2\,\dd\mathcal H^1.
 \end{aligned}
\]
Integrating this almost-everywhere inequality in $s$ and using the
coarea formula for $x\mapsto|x|$, for $R>s_0>r$, we get 
\[
 \int_{B_R(0)\setminus\overline{B_{s_0}(0)}}|w|^2\dd x
 \geq\frac{|c|^2}{2\pi}\int_{s_0}^R\frac{\dd s}{s}
 =\frac{|c|^2}{2\pi}\log\frac{R}{s_0}.
\]
Since $w\in L^2(V;\mathbb{R}^2)$, letting $R\to\infty$ we have $c=0$.
\end{proof}

The next proposition constructs a global variational comparison field for $k$ in the exterior domain.

\begin{proposition}
\label{prop:global-variational-comparison}
Let $0<\delta<\frac14$, let
$\widetilde W$, $\widetilde A$, and $\lambda$ be defined in
Lemma~\ref{lem:convex-extension}, let $r>0$, and set
$V:=\{x\in\mathbb{R}^2:|x|>r\}
$. Suppose that $k\in C^\infty(\mathcal U;\mathbb{R}^2)$ for an open
neighborhood $\mathcal U$ of $\overline V$, that
$F\in L^2(V;\mathbb{R}^2)$, and that
\begin{equation}\label{eq:global-comparison-assumptions}
 \curl k=0,
 \qquad
 \Div\bigl(\widetilde A(k)-F\bigr)=0
 \quad\text{in }\mathcal D'(V).
\end{equation}
Then there exists
$ \ell\in W^{1,2}_{\mathrm{loc}}(V;\mathbb{R}^2)
       \cap C^0(V;\mathbb{R}^2)$
such that
\begin{equation}\label{eq:global-comparator-equations}
 \curl\ell=0,
 \qquad
 \Div\widetilde A(\ell)=0
 \quad\text{in }\mathcal D'(V),
\end{equation}
and
\begin{equation}\label{eq:global-comparison-L2}
 \|k-\ell\|_{L^2(V;\mathbb{R}^2)}
 \leq\lambda^{-1}\|F\|_{L^2(V;\mathbb{R}^2)}.
\end{equation}
Moreover, $k$ and $\ell$ have the same circulation:
\begin{equation}\label{eq:global-comparator-circulation}
 \int_{\partial B_s(0)}\ell\cdot\tau\,\dd\mathcal H^1
 =\int_{\partial B_s(0)}k\cdot\tau\,\dd\mathcal H^1
 \qquad\forall\,s>r,
\end{equation}
where $\tau(x)=(-x_2,x_1)/|x|$ is the positively oriented unit tangent to $\partial B_s(0)$. Additionally, if
\begin{equation}\label{eq:global-comparison-k-zero}
 \lim_{R\to\infty}
 \sup_{\substack{x\in V\\|x|>R}}|k(x)|=0,
\end{equation}
then
\begin{equation}\label{eq:global-comparator-zero}
 \lim_{R\to\infty}
 \sup_{\substack{x\in V\\|x|>R}}|\ell(x)|=0.
\end{equation}
\end{proposition}

\begin{proof}

Let $\mathcal G(V)$ denote the closure, with respect to the
$L^2(V;\mathbb R^2)$ norm, of the subspace
$ \{\nabla\varphi:\varphi\in C_c^\infty(V)\}$. 
For $w\in L^2(V;\mathbb{R}^2)$, define the Taylor-remainder density
\begin{equation*}\label{eq:taylor-remainder-density}
 \mathcal R_k(w)
 :=\int_0^1(1-t)
 D^2\widetilde W(k-tw) w\cdot w\dd t.
\end{equation*}
The Taylor-remainder form is used because the phase field $k$ may not belong
to $L^2(V;\mathbb{R}^2)$.  Although the pointwise identity
\[
 \mathcal R_k(w)
 =\widetilde W(k-w)-\widetilde W(k)
  +\widetilde A(k)\cdot w
\]
holds, its three terms need not be integrable separately. Using the Hessian
bounds we have 
\begin{equation}\label{eq:remainder-quadratic-bounds}
 \frac{\lambda}{2}|w|^2
 \leq\mathcal R_k(w)
 \leq\frac12|w|^2
 \quad\text{a.e. in }V.
\end{equation}

Define
\begin{equation}\label{eq:global-variational-functional}
 \mathcal J(w)
 :=\int_V\bigl(\mathcal R_k(w)-F\cdot w\bigr)\dd x,
 \qquad \forall\,w\in\mathcal G(V).
\end{equation}
It follows from \eqref{eq:remainder-quadratic-bounds} that
\begin{equation}\label{eq:remainder-coercivity}
 \mathcal J(w)
 \geq\frac{\lambda}{2}\|w\|_{L^2(V;\mathbb{R}^2)}^2
      -\|F\|_{L^2(V;\mathbb{R}^2)}
       \|w\|_{L^2(V;\mathbb{R}^2)}.
\end{equation}
For almost every $x$, the Hessian of
$w\mapsto\mathcal R_k(w)(x)$ is
$D^2\widetilde W(k-w)$.  Hence the integrand is
$\lambda$-uniformly convex.  Moreover,
\[
 D_w\mathcal R_k(w)=\widetilde A(k)-\widetilde A(k-w),
 \qquad
 |D_w\mathcal R_k(w)|\leq|w|.
\]
Integration along the segment from $v$ to $w$, we obtain 
\begin{equation*}\label{eq:remainder-continuity}
 |\mathcal R_k(w)-\mathcal R_k(v)|
 \leq(|w|+|v|)|w-v|.
\end{equation*}
Note that
\[
 \left|\int_V\mathcal R_k(w)\dd x
       -\int_V\mathcal R_k(v)\dd x\right|
 \leq \bigl(
 \|w\|_{L^2(V;\mathbb{R}^2)}+
 \|v\|_{L^2(V;\mathbb{R}^2)}\bigr)
 \|w-v\|_{L^2(V;\mathbb{R}^2)}.
\]
Thus the integral of $\mathcal R_k$ is norm-continuous and convex on
$L^2(V;\mathbb{R}^2)$. Therefore
it is  weakly lower semicontinuous.  Since
$\mathcal G(V)$ is weakly closed, \eqref{eq:remainder-coercivity} and
the direct method give a unique minimizer
$w\in\mathcal G(V)$.

The Euler equation for the minimizer is obtained as follows.  If
$h\in\mathcal G(V)$ and
$0<|t|\leq1$, then
\[
 \left|
 \frac{\mathcal R_k(w+th)-\mathcal R_k(w)}{t}
 \right|
 \leq (|w|+|h|)|h|.
\]
The right-hand side, together with $|F||h|$, belongs to $L^1(V)$. Hence by
dominated convergence, we get
\begin{equation}\label{eq:variational-Euler}
 \int_V
 \bigl[\widetilde A(k)-\widetilde A(k-w)-F\bigr]\cdot h\dd x=0
 \qquad\forall\,h\in\mathcal G(V).
\end{equation}
Taking $h=w$ and using
\eqref{eq:extension-lipschitz-monotone}, we obtain
\[
 \lambda\|w\|_{L^2(V;\mathbb{R}^2)}^2
 \leq\int_V F\cdot w\dd x
 \leq\|F\|_{L^2(V;\mathbb{R}^2)}
       \|w\|_{L^2(V;\mathbb{R}^2)}.
\]
Consequently,
\begin{equation}\label{eq:variational-correction-bound}
 \|w\|_{L^2(V;\mathbb{R}^2)}
 \leq\lambda^{-1}\|F\|_{L^2(V;\mathbb{R}^2)}.
\end{equation}

Set $\ell:=k-w$. Every element of $\mathcal G(V)$ is an
$L^2(V;\mathbb{R}^2)$-limit of gradients,
and hence $\curl w=0$ in $\mathcal D'(V)$.  Thus
$\curl\ell=0$.  For $\varphi\in C_c^\infty(V)$, take
$h=\nabla\varphi$ in \eqref{eq:variational-Euler}.  Subtracting the weak
equation in \eqref{eq:global-comparison-assumptions} we obtain 
\[
 \int_V\widetilde A(\ell)\cdot\nabla\varphi\dd x=0.
\]
This proves \eqref{eq:global-comparator-equations}, and
\eqref{eq:global-comparison-L2} follows from
\eqref{eq:variational-correction-bound}.

We establish the local estimate used below.  Fix $x_0\in V$ and
$\rho>0$ such that $ \overline{B_{4\rho}(x_0)}\subset V$. 
Since $\ell\in L^2_{\mathrm{loc}}(V;\mathbb{R}^2)$ and
$\curl\ell=0$, from the weak Poincar\'e lemma we have 
$\Phi\in W^{1,2}(B_{4\rho}(x_0))$ such that
$\nabla\Phi=\ell$ almost everywhere.  Moreover,
\[
 \Div\widetilde A(\nabla\Phi)=0
 \quad\text{in }\mathcal D'(B_{4\rho}(x_0)).
\]
For $m\in\{1,2\}$ and $0<|\varepsilon|<\rho$, define
\[
 v_\varepsilon(x)
 :=\frac{\Phi(x+\varepsilon e_m)-\Phi(x)}{\varepsilon},
 \qquad \text{for a.e. }x\in B_{3\rho}(x_0).
\]
To justify the translated equation, fix
$\psi\in C_c^\infty(B_{3\rho}(x_0))$.  The functions $\psi$ and
$z\mapsto\psi(z-\varepsilon e_m)$ belong to
$C_c^\infty(B_{4\rho}(x_0))$.  Testing the equation for $\Phi$ with
these two functions, changing variables $z=x+\varepsilon e_m$ in the
second test, subtracting, and dividing by $\varepsilon$, we infer 
\[
 \begin{aligned}
 0
 &=\int_{B_{3\rho}(x_0)}
 \frac{\widetilde A(\nabla\Phi(x+\varepsilon e_m))
       -\widetilde A(\nabla\Phi(x))}{\varepsilon}
 \cdot\nabla\psi(x)\dd x\\
 &=\int_{B_{3\rho}(x_0)}
 B_\varepsilon\nabla v_\varepsilon\cdot\nabla\psi\dd x.
 \end{aligned}
\]
Consequently,
\[
 \Div(B_\varepsilon\nabla v_\varepsilon)=0
 \quad\text{in }\mathcal D'(B_{3\rho}(x_0)),
\]
where
\[
 B_\varepsilon(x)
 :=\int_0^1D\widetilde A\bigl(
 (1-t)\nabla\Phi(x)+t\nabla\Phi(x+\varepsilon e_m)\bigr)\dd t,
 \qquad \text{for a.e. }x\in B_{3\rho}(x_0).
\]
The Hessian bounds in Lemma~\ref{lem:convex-extension} imply
\[
 \lambda|\xi|^2
 \leq B_\varepsilon(x)\xi\cdot\xi
 \leq|\xi|^2
 \quad\text{for a.e. }x\in B_{3\rho}(x_0)
 \quad\forall\,\xi\in\mathbb{R}^2.
\]
Choose $\zeta\in C_c^\infty(B_{3\rho}(x_0))$ such that
\[
 0\leq\zeta\leq1,\qquad
 \zeta=1\ \text{on }B_{2\rho}(x_0),\qquad
 |\nabla\zeta|\leq\frac{C}{\rho}.
\]
Testing the equation for $v_\varepsilon$ with
$\zeta^2v_\varepsilon$,  we get the Caccioppoli-type inequality
\[
 \int_{B_{2\rho}(x_0)}|\nabla v_\varepsilon|^2\dd x
 \leq\frac{C}{\rho^2}
 \int_{B_{3\rho}(x_0)}|v_\varepsilon|^2\dd x.
\]
It is known that $\Phi$ satisfies
the standard difference-quotient identity
\cite[Section~5.8.2]{Evans}:
\[
 v_\varepsilon(x)
 =\int_0^1\partial_m\Phi(x+t\varepsilon e_m)\dd t
 \quad\text{for a.e. }x\in B_{3\rho}(x_0).
\]
Using Jensen's inequality and translation of the integration domain, we obtain 
\[
 \begin{aligned}
 \|v_\varepsilon\|_{L^2(B_{3\rho}(x_0))}^2
 &\leq\int_0^1\int_{B_{3\rho}(x_0)}
 |\partial_m\Phi(x+t\varepsilon e_m)|^2\dd x\dd t\\
 &\leq\|\ell_m\|_{L^2(B_{4\rho}(x_0))}^2
 \leq\|\ell\|_{L^2(B_{4\rho}(x_0);\mathbb{R}^2)}^2.
 \end{aligned}
\]
Combining this estimate with the preceding Caccioppoli inequality, we have 
\begin{equation}\label{eq:difference-quotient-uniform-bound}
 \int_{B_{2\rho}(x_0)}|\nabla v_\varepsilon|^2\dd x
 \leq\frac{C}{\rho^2}
 \int_{B_{4\rho}(x_0)}|\ell|^2\dd x.
\end{equation}

Let $\varepsilon_n\to0$ with $0<|\varepsilon_n|<\rho$. By 
the difference-quotient theorem, we infer 
\[
 v_{\varepsilon_n}\to
 \partial_m\Phi=\ell_m
 \quad\text{strongly in }L^2(B_{3\rho}(x_0)).
\]
After passing to a subsequence, by 
\eqref{eq:difference-quotient-uniform-bound} we have 
\[
 \nabla v_{\varepsilon_n}\rightharpoonup g_m
 \quad\text{weakly in }
 L^2(B_{2\rho}(x_0);\mathbb{R}^2)
\]
for some $g_m$.  If
$\psi\in C_c^\infty(B_{2\rho}(x_0))$ and $j\in\{1,2\}$, then
\[
 \begin{aligned}
 \int_{B_{2\rho}(x_0)}\ell_m\partial_j\psi\dd x
 &=\lim_{n\to\infty}
  \int_{B_{2\rho}(x_0)}v_{\varepsilon_n}
  \partial_j\psi\dd x\\
 &=-\lim_{n\to\infty}
  \int_{B_{2\rho}(x_0)}
  \partial_jv_{\varepsilon_n}\,\psi\dd x =-\int_{B_{2\rho}(x_0)}(g_m)_j\psi\dd x.
 \end{aligned}
\]
Thus $g_m=\nabla\ell_m$ and
$\ell_m\in W^{1,2}(B_{2\rho}(x_0))$. From weak lower semicontinuity we have 
\[
 \int_{B_{2\rho}(x_0)}|\nabla\ell_m|^2\dd x
 \leq\frac{C}{\rho^2}
 \int_{B_{4\rho}(x_0)}|\ell|^2\dd x.
\]
Summing over $m=1,2$ and absorbing the factor two into $C$, we get 
\begin{equation*}\label{eq:global-comparator-Caccioppoli}
 \int_{B_{2\rho}(x_0)}|D\ell|^2\dd x
 \leq\frac{C}{\rho^2}
 \int_{B_{4\rho}(x_0)}|\ell|^2\dd x.
\end{equation*}
Consequently, $\ell\in W^{1,2}_{\mathrm{loc}}(V;\mathbb{R}^2)$.

Since $ D\widetilde A=D^2\widetilde W \in L^\infty(\mathbb{R}^2;\mathbb{R}^{2\times2})$
and $\ell\in W^{1,2}_{\mathrm{loc}}(V;\mathbb{R}^2)$, by the Sobolev
chain rule we obtain 
\[
 \widetilde A(\ell)\in
 W^{1,2}_{\mathrm{loc}}(V;\mathbb{R}^2),\qquad
 \partial_m[\widetilde A(\ell)]
 =D\widetilde A(\ell)\partial_m\ell
 \quad\text{in }L^2_{\mathrm{loc}}(V;\mathbb{R}^2)
\]
for $m\in\{1,2\}$.  Componentwise, we have 
\[
 \partial_m[\widetilde A_i(\ell)]
 =\sum_{q=1}^2
 \partial_{\ell_q}\widetilde A_i(\ell)\,\partial_m\ell_q
 \quad\text{almost everywhere in }V,
 \qquad \forall\,i,m\in\{1,2\}.
\]
Applying $\partial_m$ to
$\Div\widetilde A(\ell)=0$ in $\mathcal D'(V)$, we obtain 
\[
 \Div\bigl(D\widetilde A(\ell)\partial_m\ell\bigr)=0
 \quad\text{in }\mathcal D'(V).
\]
Moreover, by $\curl\ell=0$ and
$\ell\in W^{1,2}_{\mathrm{loc}}(V;\mathbb{R}^2)$, we get 
\[
 \partial_m\ell_j=\partial_j\ell_m
 \quad\text{almost everywhere in }V,
 \qquad \forall\,j,m\in\{1,2\}.
\]
Hence $\partial_m\ell=\nabla\ell_m$ almost everywhere, and
\begin{equation}\label{eq:global-comparator-component-equation}
 \Div\bigl(D\widetilde A(\ell)\nabla\ell_m\bigr)=0
 \quad\text{in }\mathcal D'(V),
 \qquad \forall\,m\in\{1,2\}.
\end{equation}
For each $m\in\{1,2\}$, the preceding argument implies 
$\ell_m\in W^{1,2}(B_{2\rho}(x_0))$.  Set
$ G:=D\widetilde A(\ell)$.
This coefficient is measurable and symmetric.  More explicitly,
$ G\in L^\infty(B_{2\rho}(x_0);\mathbb{R}^{2\times2})$
and
\[
 \lambda|\xi|^2\leq G(x)\xi\cdot\xi\leq|\xi|^2
 \quad\text{for a.e. }x\in B_{2\rho}(x_0)
 \quad\forall\,\xi\in\mathbb{R}^2.
\]
Equation \eqref{eq:global-comparator-component-equation} is precisely
the weak scalar equation required in
Lemma~\ref{lem:interior-elliptic-estimates}.
Applying \eqref{eq:interior-DGN-estimate} of Lemma \ref{lem:interior-elliptic-estimates}, we can find some
$\alpha=\alpha(\lambda)\in(0,1)$ such that,
\begin{equation}\label{eq:global-comparator-local-bound}
 \|\ell\|_{L^\infty(B_\rho(x_0);\mathbb{R}^2)}
 +\rho^\alpha
 [\ell]_{C^{0,\alpha}(B_\rho(x_0);\mathbb{R}^2)}
 \leq\frac{C}{\rho}
 \|\ell\|_{L^2(B_{2\rho}(x_0);\mathbb{R}^2)}.
\end{equation}
The constant $C$ in \eqref{eq:global-comparator-local-bound} depends
only on $\lambda$.

Thus $ w=k-\ell\in W^{1,2}_{\mathrm{loc}}(V;\mathbb{R}^2)
   \cap L^2(V;\mathbb{R}^2)$
satisfies $\curl w=0$ in $\mathcal D'(V)$.
Since $k$ and $\ell$ are continuous, the Sobolev trace of $w$ on each
$\partial B_s(0)$ coincides with its pointwise restriction. Therefore, using
Lemma~\ref{lem:weak-circulation} we know that the circulation
of $w$ is zero on every such circle. Thus, $k$ and $\ell$ have the same
circulation, which proves
\eqref{eq:global-comparator-circulation}.

Finally, assume \eqref{eq:global-comparison-k-zero}.  Let $R>r+2$, and let
$x_0\in\mathbb{R}^2$ satisfy $|x_0|>R+2$.  Then
$\overline{B_4(x_0)}\subset V$ and
$B_2(x_0)\subset\{y\in\mathbb{R}^2:|y|>R\}$.  Applying
\eqref{eq:global-comparator-local-bound} with $\rho=1$, we obtain 
\[
 \sup_{\substack{x\in V\\|x|>R+2}}|\ell(x)|
 \leq C\Big(
 \mathcal L^2(B_2(0))^{1/2}
 \sup_{\substack{y\in V\\|y|>R}}|k(y)|
 +\|w\|_{L^2(\{y\in\mathbb{R}^2:|y|>R\};\mathbb{R}^2)}
 \Big).
\]
Therefore, the first term tends to zero by
\eqref{eq:global-comparison-k-zero}, and the second tends to zero
because $w\in L^2(V;\mathbb{R}^2)$.  This proves
\eqref{eq:global-comparator-zero}.
\end{proof}

\begin{lemma}
\label{lem:planar-puncture-removal}
Let $R>0$ and $N\in\mathbb N^{+}$.  Suppose that
\[
 v=(v_1,\ldots,v_N)\in W^{1,2}_{\mathrm{loc}}
 \bigl(B_R(0)\setminus\{0\};\mathbb{R}^N\bigr)
 \cap L^{\infty}(B_{R/2}(0);\mathbb{R}^N)
\]
and
$Q\in L^\infty(B_R(0);\mathbb{R}^{2\times2})$ is symmetric.  Assume
that there are constants $0<\lambda\leq\Lambda<\infty$ such that
\[
 \lambda|\xi|^2
 \leq Q(y)\xi\cdot\xi
 \leq\Lambda|\xi|^2
 \quad\text{for a.e. }y\in B_R(0)
 \quad\forall\,\xi\in\mathbb{R}^2.
\]
If
\[
 \Div(Q\nabla v_m)=0
 \quad\text{in }\mathcal D'(B_R(0)\setminus\{0\}),
 \qquad \forall\,m\in\{1,\ldots,N\},
\]
then $ v\in W^{1,2}_{\mathrm{loc}}(B_R(0);\mathbb{R}^N)$ and
\[
 \Div(Q\nabla v_m)=0
 \quad\text{in }\mathcal D'(B_R(0)),
 \qquad \forall\,m\in\{1,\ldots,N\}.
\]
\end{lemma}

\begin{proof}
Fix $\rho>0$ such that $4\rho<R$, and choose
$\eta\in C_c^\infty(B_{2\rho}(0))$ that satisfies
\[
 0\leq\eta\leq1,\qquad
 \eta=1\quad\text{on }B_\rho(0),\qquad
 |\nabla\eta|\leq\frac{C}{\rho}.
\]
For $0<\varepsilon<\min\{\rho,\frac12\}$, choose a radial cutoff
$\chi_\varepsilon\in C^\infty(\mathbb R^2)$ such that
\[
 0\leq\chi_\varepsilon\leq1,\qquad
 \chi_\varepsilon=0
 \quad\text{on }\overline{B_{\varepsilon^2}(0)},\qquad
 \chi_\varepsilon=1
 \quad\text{on }\mathbb R^2\setminus B_\varepsilon(0),
\]
and
\[
 \supp(\nabla\chi_\varepsilon)
 \subset
 B_\varepsilon(0)\setminus\overline{B_{\varepsilon^2}(0)},
 \qquad
 |\nabla\chi_\varepsilon(y)|
 \leq\frac{C}{|y|\,|\log\varepsilon|}
 \quad\text{for }
 \varepsilon^2<|y|<\varepsilon.
\]
Consequently
\begin{equation}\label{eq:planar-log-cutoff}
 \int_{B_\varepsilon(0)}|\nabla\chi_\varepsilon|^2\dd y
 \leq\frac{C}{|\log\varepsilon|},
 \qquad
 \int_{B_\varepsilon(0)}|\nabla\chi_\varepsilon|\dd y
 \leq\frac{C\varepsilon}{|\log\varepsilon|}.
\end{equation}
For example, one may choose
$\theta\in C^\infty(\mathbb R;[0,1])$ such that
\[
 \theta=0\quad\text{on }(-\infty,1/4],
 \qquad
 \theta=1\quad\text{on }[3/4,+\infty),
\]
and define
\[
 \chi_\varepsilon(0)=0,
 \qquad
 \chi_\varepsilon(y)
 =
 \theta\left(
 \frac{\log |y|-2\log\varepsilon}{|\log\varepsilon|}
 \right)
 \quad\text{for }y\neq0.
\]

For each $m$, the function
$\eta^2\chi_\varepsilon^2v_m$ belongs to
$
 W^{1,2}_0\bigl(
 B_{2\rho}(0)\setminus\overline{B_{\varepsilon^2}(0)}
 \bigr)$. 
 Testing the
component equations, summing over $m$, and using ellipticity and
Young's inequality, we have 
\begin{equation}\label{eq:puncture-Caccioppoli}
 \int_{B_{2\rho}(0)}
 \eta^2\chi_\varepsilon^2|\nabla v|^2\dd y
 \leq C\int_{B_{2\rho}(0)}|v|^2
 \bigl(
 \chi_\varepsilon^2|\nabla\eta|^2
 +\eta^2|\nabla\chi_\varepsilon|^2
 \bigr)\dd y.
\end{equation}
The term containing $\nabla\eta$ is bounded independently of
$\varepsilon$.  Since $v\in L^{\infty}(B_{R/2}(0);\mathbb{R}^N)$, we get 
\[
 \int_{B_{2\rho}(0)}|v|^2\eta^2
 |\nabla\chi_\varepsilon|^2\dd y
 \leq\frac{C}{|\log\varepsilon|} \|v\|_{L^{\infty}(B_\varepsilon(0);\mathbb{R}^N)}^2 \to0\qquad \mbox{as}~\varepsilon\to 0^{+}.
\]
For every
$y\in B_{2\rho}(0)\setminus\{0\}$,
$\chi_\varepsilon(y)\to1$. Using Fatou's lemma to
\eqref{eq:puncture-Caccioppoli}, we obtain 
\begin{equation}\label{eq:puncture-gradient-bound}
 \int_{B_\rho(0)}|\nabla v|^2\dd y<\infty.
\end{equation}

Let $g_{jm}$ denote the weak derivative $\partial_jv_m$ on the
punctured disk, arbitrarily extended at the origin.  By
\eqref{eq:puncture-gradient-bound}, $g_{jm}\in L^2(B_\rho(0))$.
For $\varphi\in C_c^\infty(B_\rho(0))$, integration by parts with
$\chi_\varepsilon\varphi$, justified by density, we get 
\[
 \int_{B_\rho(0)}
 \chi_\varepsilon v_m\partial_j\varphi\dd y
 +\int_{B_\rho(0)}
 v_m\varphi\partial_j\chi_\varepsilon\dd y
 =-\int_{B_\rho(0)}
 \chi_\varepsilon g_{jm}\varphi\dd y.
\]
The middle term satisfies
\[
 \left|
 \int_{B_\rho(0)}v_m\varphi
 \partial_j\chi_\varepsilon\dd y
 \right|
 \leq\|\varphi\|_{L^\infty(B_\rho(0))}
 \|v\|_{L^{\infty}(B_\varepsilon(0);\mathbb{R}^N)}
 \|\nabla\chi_\varepsilon\|_
 {L^1(B_\varepsilon(0);\mathbb{R}^2)}
 \to0.
\]
Assign an arbitrary value to $v$ at the origin, and denote the resulting extension again by $v$. Since a single point has measure zero and $v\in L^\infty(B_\rho(0))$, we have $v\in L^2(B_\rho(0))$.  Dominated convergence
in terms containing $\chi_\varepsilon$, together with the preceding
estimate for the term containing $\nabla\chi_\varepsilon$, we obtain 
\[
\int_{B_\rho(0)} v_m\,\partial_j\varphi\dd y
=
-\int_{B_\rho(0)} g_{jm}\,\varphi\dd y
\qquad
\text{for all }\varphi\in C_c^\infty(B_\rho(0)).
\]
Thus, $\partial_j v_m=g_{jm}$ in $\mathcal D'(B_\rho(0))$. Since
$v_m,g_{jm}\in L^2(B_\rho(0))$, it follows that $v\in W^{1,2}(B_\rho(0);\mathbb R^N)$.

Finally, let $\psi\in C_c^\infty(B_\rho(0))$.  Testing the equation on
the punctured disk with $\chi_\varepsilon\psi$, we have 
\[
 0=
 \int_{B_\rho(0)}
 \chi_\varepsilon Q\nabla v_m\cdot\nabla\psi\dd y
 +\int_{B_\rho(0)}
 \psi Q\nabla v_m\cdot\nabla\chi_\varepsilon\dd y.
\]
The second term tends to zero because
\[
 \begin{aligned}
 &\left|
 \int_{B_\rho(0)}
 \psi Q\nabla v_m\cdot\nabla\chi_\varepsilon\dd y
 \right|\\
 &\quad\leq C\|\psi\|_{L^\infty(B_\rho(0))}
 \|\nabla v_m\|_{L^2(B_\varepsilon(0);\mathbb{R}^2)}
 \|\nabla\chi_\varepsilon\|_
 {L^2(B_\varepsilon(0);\mathbb{R}^2)}
 \to0.
 \end{aligned}
\]
Here, \eqref{eq:puncture-gradient-bound} provides a uniform bound for
$\|\nabla v_m\|_{L^2(B_\varepsilon(0);\mathbb{R}^2)}$, while
\eqref{eq:planar-log-cutoff} gives
$\|\nabla\chi_\varepsilon\|_{L^2(B_\varepsilon(0);\mathbb{R}^2)}
\leq C|\log\varepsilon|^{-1/2}\to0$.
Moreover,
\[
 \begin{aligned}
 &\left|
 \int_{B_\rho(0)}(1-\chi_\varepsilon)
 Q\nabla v_m\cdot\nabla\psi\dd y
 \right|\\
 &\quad\leq
 C\|\nabla\psi\|_{L^\infty(B_\rho(0);\mathbb{R}^2)}
 \mathcal L^2(B_\varepsilon(0))^{1/2}
 \|\nabla v_m\|_{L^2(B_\varepsilon(0);\mathbb{R}^2)}
 \to0.
 \end{aligned}
\]
It follows that the weak equation holds in $B_\rho(0)$.  It already
holds in $B_R(0)\setminus\{0\}$.  Given
$\psi\in C_c^\infty(B_R(0))$, a smooth partition of unity subordinate
to
\[
 B_R(0)=B_\rho(0)\cup\bigl(B_R(0)\setminus\{0\}\bigr)
\]
decomposes $\psi$ into test functions supported in these two sets.
This proves both assertions on $B_R(0)$.
\end{proof}

The next lemma gives the decay estimate for the homogeneous field.

\begin{lemma}
\label{lem:critical-decay}
Let $R>0$, set $E_R:=\{x\in\mathbb{R}^2:|x|>R\}$, and suppose that
$\ell\in W^{1,2}_{\mathrm{loc}}(E_R;\mathbb{R}^2)
       \cap C^0(E_R;\mathbb{R}^2)$
satisfies
\begin{equation*}\label{eq:critical-decay-system}
 \curl\ell=0,
 \qquad
 \Div\bigl((1-|\ell|^2)\ell\bigr)=0
 \quad\text{in }\mathcal D'(E_R),
\end{equation*}
and
\begin{equation}\label{eq:critical-decay-zero}
 \lim_{S\to\infty}
 \sup_{\substack{x\in E_R\\|x|>S}}|\ell(x)|=0.
\end{equation}
Then there exist constants $C>0$ and $R_1>R$ such that
\begin{equation}\label{eq:critical-decay-rate}
 |\ell(x)|\leq \frac{C}{|x|}
 \qquad
 \forall\,x\in\mathbb{R}^2\text{ with }|x|>R_1.
\end{equation}
In particular, $\ell\in L^\gamma(E_{R_1};\mathbb{R}^2)$ for all $\gamma>2$.
\end{lemma}

\begin{proof}
By \eqref{eq:critical-decay-zero}, choose $R_0>R$ such that
\[
 \sup_{\substack{x\in E_R\\|x|>R_0}}|\ell(x)|\leq\frac{1}{4}.\]
Set
\[
 M(x):=DA(\ell(x))
 =(1-|\ell(x)|^2)I-2\ell(x)\otimes\ell(x),
 \qquad \forall\,x\in E_{R_0}.
\]
The eigenvalues of $M$ are $1-|\ell|^2$ and
$1-3|\ell|^2$.  Consequently,
\begin{equation}\label{eq:critical-ellipticity}
 \frac{13}{16}|\xi|^2
 \leq M(x)\xi\cdot\xi
 \leq |\xi|^2
 \qquad
 \forall\,x\in E_{R_0},\ \forall\,\xi\in\mathbb{R}^2.
\end{equation}
Using the Sobolev chain rule we get 
$A(\ell)\in W^{1,2}_{\mathrm{loc}}(E_{R_0};\mathbb{R}^2)$ and
\[
 \partial_m[A(\ell)]=DA(\ell)\partial_m\ell
 \quad\text{in }L^2_{\mathrm{loc}}(E_{R_0};\mathbb{R}^2).
\]
Applying $\partial_m$ to the divergence equation and using
$\partial_m\ell_j=\partial_j\ell_m$ almost everywhere, we obtain
\begin{equation}\label{eq:critical-component-equation}
 \Div(M\nabla\ell_m)=0
 \quad\text{in }\mathcal D'(E_{R_0}),
 \qquad \forall\,m\in\{1,2\}.
\end{equation}

Apply the Kelvin transform to
\eqref{eq:critical-component-equation}.  For $y\neq0$, we define
\[
 \mathcal K(y):=\frac{y}{|y|^2},
 \qquad
 \mathcal R(y):=I-2\frac{y}{|y|}\otimes\frac{y}{|y|}.
\]
For $0<|y|<R_0^{-1}$, set
\[
 v(y):=\ell(\mathcal K(y)),
 \qquad
 \widehat M(y):=\mathcal R(y)M(\mathcal K(y))\mathcal R(y).
\]
On every compact subset of
$B_{1/R_0}(0)\setminus\{0\}$, the maps $\mathcal K$ and
$\mathcal K^{-1}$ are smooth with bounded derivatives.  Thus by the Sobolev
composition theorem, we infer 
\[
 v\in W^{1,2}_{\mathrm{loc}}
 \bigl(B_{1/R_0}(0)\setminus\{0\};\mathbb{R}^2\bigr).
\]
Note that we have the fllowing identities
\[ D\mathcal K(y)=|y|^{-2}\mathcal R(y),\quad
 \mathcal R(y)^2=I,\quad
 |\det D\mathcal K(y)|=|y|^{-4},
 \quad\forall\,y\in\mathbb{R}^2\setminus\{0\}.\]
 Fix $m\in\{1,2\}$, let
$\psi\in C_c^\infty(B_{1/R_0}(0)\setminus\{0\})$, and define
$\varphi:=\psi\circ\mathcal K$.  Then
$ \supp\varphi=\mathcal K(\supp\psi)\Subset E_{R_0}$, 
so $\varphi\in C_c^\infty(E_{R_0})$.  Since $\mathcal K$ is an
involution and $\mathcal R(\mathcal K(y))=\mathcal R(y)$, from the Sobolev
chain rule, for almost every
$y\in B_{1/R_0}(0)\setminus\{0\}$, we have 
\[
 \nabla_x\ell_m(\mathcal K(y))
 =|y|^2\mathcal R(y)\nabla_yv_m(y),
 \qquad
 \nabla_x\varphi(\mathcal K(y))
 =|y|^2\mathcal R(y)\nabla_y\psi(y).
\]
Changing variables $x=\mathcal K(y)$ and substituting the two chain
rules we infer 
\[
 \begin{aligned}
 0
 &=\int_{E_{R_0}}M\nabla\ell_m\cdot\nabla\varphi\dd x\\
 &=\int_{B_{1/R_0}(0)\setminus\{0\}}
 M(\mathcal K(y))
 \bigl(|y|^2\mathcal R(y)\nabla v_m(y)\bigr)
 \cdot
 \bigl(|y|^2\mathcal R(y)\nabla\psi(y)\bigr)
 |y|^{-4}\dd y\\
 &=\int_{B_{1/R_0}(0)\setminus\{0\}}
 \mathcal R(y)M(\mathcal K(y))\mathcal R(y)
 \nabla v_m(y)\cdot\nabla\psi(y)\dd y\\
 &=\int_{B_{1/R_0}(0)\setminus\{0\}}
 \widehat M(y)\nabla v_m(y)\cdot\nabla\psi(y)\dd y.
 \end{aligned}
\]
Since this identity holds for every
$\psi\in C_c^\infty(B_{1/R_0}(0)\setminus\{0\})$,
\begin{equation*}\label{eq:critical-inverted-equation}
 \Div(\widehat M\nabla v_m)=0
 \quad\text{in }\mathcal D'(B_{1/R_0}(0)\setminus\{0\}),
 \qquad \forall\,m\in\{1,2\}.
\end{equation*}
The matrix $\mathcal R(y)$ is orthogonal, so $\widehat M$ has the
ellipticity bounds in \eqref{eq:critical-ellipticity}.  The limit in
\eqref{eq:critical-decay-zero} shows that the definition $v(0):=0$
makes $v$ continuous on $B_{1/R_0}(0)$.  Set $\widehat M(0):=I$. Applying
Lemma~\ref{lem:planar-puncture-removal}  with
$N=2$, $R=R_0^{-1}$, and $Q=\widehat M$, we obtain 
\begin{equation}\label{eq:critical-full-disk-equation}
 \begin{gathered}
  v\in W^{1,2}_{\mathrm{loc}}(B_{1/R_0}(0);\mathbb{R}^2),\\
  \Div(\widehat M\nabla v_m)=0
  \quad\text{in }\mathcal D'(B_{1/R_0}(0)),
  \quad\forall\,m\in\{1,2\}.
 \end{gathered}
\end{equation}

Set
\[
 r_*:=\frac18\min\{R_0^{-1},1\}.
\]
Then $\overline{B_{4r_*}(0)}\subset B_{1/R_0}(0)$.  For each
$m\in\{1,2\}$, by \eqref{eq:critical-full-disk-equation} we have 
$v_m\in W^{1,2}(B_{4r_*}(0))$.  The associated coefficient matrix $G=\widehat M$ is measurable, symmetric, and belongs to
$L^\infty(B_{4r_*}(0);\mathbb{R}^{2\times2})$. Applying Lemma \ref{lem:interior-elliptic-estimates} with $r=2r_*$ to each component, we obtain $\beta\in(0,1)$ and a finite constant $C$, depending only on the
ellipticity constants in \eqref{eq:critical-ellipticity}, such that
\[
 |v_m(y)-v_m(z)|
 \leq Cr_*^{-1-\beta}
 \|v_m\|_{L^2(B_{4r_*}(0))}|y-z|^\beta
\]
for $y,z\in B_{2r_*}(0)$ and $m\in\{1,2\}$.  Since $v(0)=0$, after
enlarging the constant, we have
\begin{equation}\label{eq:critical-initial-holder}
 |v(y)-v(z)|\leq C|y-z|^\beta,
 \qquad
 |v(y)|\leq C|y|^\beta,
 \qquad\forall\,y,z\in B_{2r_*}(0).
\end{equation}

We now verify that
$\widehat M\in C^{0,\beta}
(B_{r_*}(0);\mathbb{R}^{2\times2})$, which is required to derive the Schauder estimate.  Since 
\[
 DA(a)-I=-|a|^2I-2a\otimes a
 \quad\forall\,a\in\mathbb{R}^2
\]
we have 
\[
 |DA(a)-I|\leq C|a|^2,
 \qquad
 |DA(a)-DA(b)|\leq C(|a|+|b|)|a-b|
 \quad\forall\,a,b\in\mathbb{R}^2.
\]
For $y\neq0$,
\begin{equation}\label{eq:critical-inverted-structure}
 \widehat M(y)-I
 =\mathcal R(y)\bigl(DA(v(y))-I\bigr)\mathcal R(y).
\end{equation}
In particular, $\widehat M(y)\to I=\widehat M(0)$ as $y\to0$.

Let $y,z\in B_{r_*}(0)$, and set
$d:=|y-z|$ and $r:=\max\{|y|,|z|\}$.  If $d\geq r/2$, then by 
\eqref{eq:critical-initial-holder} and
\eqref{eq:critical-inverted-structure} we have 
\[
 |\widehat M(y)-\widehat M(z)|
 \leq Cr^{2\beta}\leq Cd^{2\beta}\leq Cd^\beta.
\]
Here $d\leq2r_*\leq1/4$, so the last inequality follows from
$0<\beta<1$.
Suppose that $d<r/2$.  Then $y,z\neq0$ and
\[
 \frac r2<\min\{|y|,|z|\}
 \leq\max\{|y|,|z|\}=r,
 \qquad
 |\mathcal R(y)-\mathcal R(z)|\leq C\frac d r.
\]
The exact decomposition
\[
\begin{aligned}
 \widehat M(y)-\widehat M(z)
 ={}&\mathcal R(y)\bigl(DA(v(y))-DA(v(z))\bigr)\mathcal R(y)\\
 &+\bigl(\mathcal R(y)-\mathcal R(z)\bigr)
 \bigl(DA(v(z))-I\bigr)\mathcal R(y)\\
 &+\mathcal R(z)\bigl(DA(v(z))-I\bigr)
 \bigl(\mathcal R(y)-\mathcal R(z)\bigr)
\end{aligned}
\]
and \eqref{eq:critical-initial-holder} imply that
\[
 |\widehat M(y)-\widehat M(z)|
 \leq Cr^\beta d^\beta+Cd r^{2\beta-1}
 \leq Cd^\beta,
\]
because
$d r^{2\beta-1}
 =d^\beta(d/r)^{1-\beta}r^\beta\leq d^\beta$ and
$r^\beta\leq1$.
This proves
\begin{equation}\label{eq:critical-coefficient-holder}
 \widehat M\in
 C^{0,\beta}(B_{r_*}(0);\mathbb{R}^{2\times2}).
\end{equation}

For each $m\in\{1,2\}$,
$v_m\in W^{1,2}(B_{r_*}(0))$ is a weak solution.  The symmetry and
ellipticity bounds remain valid, while
\eqref{eq:critical-coefficient-holder} satisfies the additional
hypothesis of the coefficient $C^{0,\beta}$.  Apply
\eqref{eq:interior-Schauder-estimate} with $x_0=0$ and $r=r_*/2$.  Then $ v_m\in C^{1,\beta}(B_{r_*/2}(0))
$
and
\[
 \frac{r_*}{2}
 \|\nabla v_m\|_{L^\infty(B_{r_*/2}(0);\mathbb{R}^2)}
 \leq C\|v_m\|_{L^\infty(B_{r_*}(0))}.
\]
Consequently, $\sup_{B_{\frac{r_*}{2}}(0)}|\nabla v|<\infty$. Since $v(0)=0$, integration along the segment from $0$ to $y$ we get 
\[ 
 |v(y)|\leq C|y|
 \qquad\forall\,y\in B_{r_*/2}(0).
\]
Returning through the inversion, we conclude that
\[
 |\ell(x)|=|v(\mathcal K(x))|
 \leq C|x|^{-1}
 \qquad
 \forall\,x\in\mathbb{R}^2\text{ with }
 |x|>\frac{2}{r_*}.
\]
Taking $R_1:=2/r_*$ we prove \eqref{eq:critical-decay-rate}.  The
integrability assertion follows from 
\[
 \int_{\{x\in\mathbb{R}^2:|x|>R_1\}}|x|^{-\gamma}\dd x
 =2\pi\int_{R_1}^\infty r^{1-\gamma}\dd r<\infty
 \qquad\forall\,\gamma>2.
\]
\end{proof}

\begin{proof}[Proof of Theorem~\ref{thm:exterior-phase}]
Fix $0<\delta<\frac14$, and let
$\widetilde W$, $\widetilde A$, and $\lambda$ be given by
Lemma~\ref{lem:convex-extension}.  By
\eqref{eq:k-uniform-zero}, there exists $r_\delta>r_0$ such that
\[
 |k|\leq\delta
 \quad\text{on }
 V:=\{x\in\mathbb{R}^2:|x|>r_\delta\}.
\]
Thus $\widetilde A(k)=A(k)$ on $V$, and
\eqref{eq:phase-lemma-assumptions} becomes
\[
 \Div\bigl(\widetilde A(k)-F\bigr)=0
 \quad\text{in }\mathcal D'(V).
\]
Proposition~\ref{prop:global-variational-comparison} provides a field $\ell$ such
that
\begin{equation}\label{eq:theorem-global-comparison}
 \begin{gathered}
  \curl\ell=0,\qquad
  \Div\widetilde A(\ell)=0,\qquad
  k-\ell\in L^2(V;\mathbb{R}^2),\\
  \displaystyle
  \lim_{R\to\infty}
  \sup_{\substack{x\in V\\|x|>R}}|\ell(x)|=0.
 \end{gathered}
\end{equation}

Choose $R_2>r_\delta$ such that
\[
 |\ell(x)|\leq\delta
 \qquad
 \forall\,x\in\mathbb{R}^2\text{ with }|x|>R_2.
\]
Then $\widetilde A(\ell)=A(\ell)$ on
$E_{R_2}:=\{x\in\mathbb{R}^2:|x|>R_2\}$.  Lemma~\ref{lem:critical-decay}
gives constants $C>0$ and $R_1>R_2$ such that
\[
 \begin{gathered}
  |\ell(x)|\leq C|x|^{-1}
  \quad\forall\,x\in\mathbb{R}^2\text{ with }|x|>R_1,\\
  \ell\in L^\gamma(E_{R_1};\mathbb{R}^2)
  \quad\forall\,\gamma>2.
 \end{gathered}
\]

Fix $\gamma>2$ and set $w:=k-\ell$.  By
\eqref{eq:theorem-global-comparison},
$w\in L^2(V;\mathbb{R}^2)$.  Moreover, by
\eqref{eq:k-uniform-zero} and \eqref{eq:theorem-global-comparison}, we obtain 
\[
 \begin{aligned}
 \sup_{\substack{x\in V\\|x|>R}}|w(x)|
 &\leq
 \sup_{\substack{x\in V\\|x|>R}}|k(x)|
 +\sup_{\substack{x\in V\\|x|>R}}|\ell(x)|
 \to0
 \end{aligned}
 \qquad\text{as }R\to\infty.
\]
Choose $R_3\geq R_1$ such that
\[
 |w(x)|\leq1
 \quad\forall\,x\in\mathbb{R}^2\text{ with }|x|>R_3.
\]
Set
\[
 E_{R_3}:=\{x\in\mathbb{R}^2:|x|>R_3\}.
\]
Then
\[
 \int_{E_{R_3}}|w|^\gamma\dd x
 \leq\int_{E_{R_3}}|w|^2\dd x<\infty.
\]
Therefore $k=\ell+w$ belongs to
$L^\gamma(E_{R_3};\mathbb{R}^2)$.  The set
\[
 E\setminus E_{R_3}
 =\{x\in\mathbb{R}^2:r_0<|x|\leq R_3\}
\]
has compact closure in $\mathcal U$.  Since $k$ is smooth on
$\mathcal U$, it is bounded on $E\setminus E_{R_3}$.
Therefore $k\in L^\gamma(E;\mathbb{R}^2)$.  Because $\gamma>2$ was
arbitrary, we obtain \eqref{eq:phase-prop-conclusion}.
\end{proof}

\begin{remark}\label{rem:phase-sharp}
The condition $\gamma>2$ cannot be relaxed.  Let
$C_0\in\mathbb{R}\setminus\{0\}$ and set
\[
 T(x):=\frac{(-x_2,x_1)}{|x|^2},
 \qquad
 k:=C_0T,
 \qquad
 F:=0.
\]
Then $\curl k=0$, $|k(x)|=|C_0|/|x|$, and
\[
 \int_{\partial B_s(0)}k\cdot\tau\,\dd\mathcal H^1
 =2\pi C_0
 \qquad\forall\,s>0,
\]
while
\[
 \Div\bigl((1-|k|^2)k\bigr)=0
 \quad\text{on }
 \mathbb{R}^2\setminus\overline{B_r(0)}
 \quad\forall\,r>0.
\]
Moreover,
\[
 \int_{\{x\in\mathbb{R}^2:|x|>r_0\}}|k(x)|^\gamma\dd x
 =2\pi|C_0|^\gamma\int_{r_0}^\infty r^{1-\gamma}\dd r,
\]
which is finite if and only if $\gamma>2$.  Thus the conclusion fails
in general for $\gamma\leq2$, and the $|x|^{-1}$ decay rate in
Lemma~\ref{lem:critical-decay} cannot be improved without additional
hypotheses.
\end{remark}

\section{The Ginzburg--Landau equation}
\label{sec:ginzburg-landau}

\subsection{Pointwise estimates and the Bernstein bound}
\label{subsec:pointwise-bernstein}

\begin{lemma}\label{lem:basic-bounds}
Let $u\in C^\infty(\mathbb{R}^2;\mathbb{R}^2)$ satisfy
\begin{equation*}\label{eq:basic-bounds-assumptions}
 -\Delta u=u(1-|u|^2)\quad\text{in }\mathbb{R}^2,
 \qquad
 \lim_{|x|\to\infty}|u(x)|=1.
\end{equation*}
Then
\begin{equation*}\label{eq:bernstein}
 |\nabla u(x)|^2\leq1-|u(x)|^2
 \qquad\forall\,x\in\mathbb{R}^2.
\end{equation*}
\end{lemma}

\begin{proof}
See \cite[Theorem~3.5]{Smyrnelis}.
\end{proof}

\subsection{The exterior phase}
\label{subsec:exterior-phase-gl}
For the remainder of this section, let $u\in C^\infty(\mathbb{R}^2;\mathbb{R}^2)$ satisfy
\eqref{eq:GL} and \eqref{eq:modulus-limit}.

Let
\begin{equation*}\label{eq:J-def}
 J\colon\mathbb{R}^2\to\mathbb{R}^2,
 \qquad
 J(y_1,y_2):=(-y_2,y_1).
\end{equation*}
Thus $J^2=-I$, $J^\top=-J$, and $|Jy|=|y|$.  Under the
identification $\mathbb{R}^2\simeq\mathbb{C}$, the map $J$ is multiplication by $i$.
Accordingly, if $n$ is $\mathbb{R}^2$-valued, the notation
\begin{equation*}\label{eq:n-perp-def}
 n^\perp:=Jn
\end{equation*}
means the real vector obtained by rotating $n$ through $\pi/2$.

By \eqref{eq:modulus-limit}, there exists $R_0>2$ such that
\begin{equation*}\label{eq:rho-lower-neighborhood}
 |u(x)|\geq\frac12
 \quad\forall\,x\in\mathbb{R}^2\text{ with }|x|>R_0-1.
\end{equation*}
Set
\begin{equation}\label{eq:exterior-domains}
 \mathcal U_0:=\{x\in\mathbb{R}^2:|x|>R_0-1\},
 \qquad
 \Omega:=\{x\in\mathbb{R}^2:|x|>R_0\}.
\end{equation}
Then $\mathcal U_0$ is an open neighborhood of $\overline\Omega$.  For
$0<r<R$, set
\begin{equation*}\label{eq:annulus-phase-def}
 \Acal(r,R):=\{x\in\mathbb{R}^2:r<|x|<R\}.
\end{equation*}
Define
\begin{equation}\label{eq:rho-n-def}
 \rho:=|u|,
 \qquad
 n:=\frac{u}{\rho}.
\end{equation}
Since $u$ is smooth and does not vanish on $\mathcal U_0$, one has
\[
 \rho\in C^\infty(\mathcal U_0),\qquad
 n\in C^\infty(\mathcal U_0;\mathbb{R}^2),\qquad
 n(\mathcal U_0)\subset S^1.
\]
In particular,
\begin{equation}\label{eq:rho-lower}
 \rho\geq\frac12
 \qquad\text{on }\mathcal U_0.
\end{equation}

\begin{lemma}
\label{lem:exterior-phase-decomposition}
For $a\in\{1,2\}$, define
\begin{equation}\label{eq:k-global}
 k_a:=n^\perp\cdot\partial_a n=(Jn)\cdot\partial_a n
 \qquad\text{on }\mathcal U_0,
\end{equation}
and set
$k:=(k_1,k_2)\colon\mathcal U_0\to\mathbb{R}^2$.
Then the following assertions hold.

\begin{enumerate}
\item For $a=1,2$,
\begin{equation}\label{eq:n-derivative-identities}
 \partial_a n=k_a n^\perp,
 \qquad
 \partial_a n^\perp=-k_a n
 \qquad\text{on }\mathcal U_0.
\end{equation}

\item 
\begin{equation}\label{eq:k-curl}
 \curl k=\partial_1k_2-\partial_2k_1=0
 \qquad\text{pointwise on }\mathcal U_0.
\end{equation}

\end{enumerate}
\end{lemma}

\begin{proof}
Since $|n|^2=1$, we have 
$n\cdot\partial_a n=0$.  The pair $(n,Jn)=(n,n^\perp)$ is an oriented
orthonormal basis of the target.  Therefore, $\partial_a n$ has no component in the direction of  $n$, and its $n^\perp$ component is 
$(Jn)\cdot\partial_a n=k_a$.  This proves the first identity in
\eqref{eq:n-derivative-identities}.  Applying $J$ and using $J^2=-I$
we obtain the second one.

Let $U\subset\mathcal U_0$ be a simply connected domain.  The map
\[
 p\colon\mathbb{R}\to S^1,
 \qquad p(t):=(\cos t,\sin t),
\]
is the universal covering map of $S^1$. Applying the covering-space lifting
theorem to $n|_U\colon U\to S^1$, we get a continuous lift
$\phi\colon U\to\mathbb{R}$ satisfying $p\circ\phi=n$.  Because $p$ is a local
$C^\infty$ diffeomorphism and $n$ is smooth, we have $\phi\in C^\infty(U)$.
Differentiating $n=(\cos\phi,\sin\phi)$ we infer 
$ \partial_a n=(\partial_a\phi)n^\perp$. Comparison with \eqref{eq:n-derivative-identities} we have 
$k_a=\partial_a\phi$ for $a=1,2$, and $k=\nabla\phi$ on $U$. Any two lifts differ by a locally constant multiple of $2\pi$.
Since $U$ is connected, this multiple is constant.
Finally, equality of mixed derivatives gives
\[
\curl k= \partial_1k_2-\partial_2k_1
 =\partial_{12}\phi-\partial_{21}\phi=0
 \quad\text{on }U.
\]
Every point of $\mathcal U_0$ has a simply connected neighborhood, so
\eqref{eq:k-curl} follows on all of $\mathcal U_0$.
\end{proof}

Set
\begin{equation*}\label{eq:h-e-def}
 h:=1-|u|^2,
 \qquad
 e:=|\nabla u|^2.
\end{equation*}
The Bernstein estimate gives
\begin{equation}\label{eq:h-e-bounds}
 0\leq e\leq h\leq1
 \qquad\text{on }\mathcal U_0.
\end{equation}

\begin{lemma}
\label{lem:polar-equations}
Suppose that the functions $\rho$ and $n$ defined in \eqref{eq:rho-n-def} and the
field $k$ defined in \eqref{eq:k-global}. Then pointwise on
$\mathcal U_0$,
\begin{equation}\label{eq:polar-first-derivative}
 \partial_a u=(\partial_a\rho)n+\rho k_a n^\perp,
 \qquad\forall\,a\in\{1,2\}.
\end{equation}
Moreover,
\begin{equation}\label{eq:amplitude-phase}
 -\Delta\rho+\rho|k|^2=\rho h,
 \qquad
 \Div(\rho^2k)=0
 \qquad\text{on }\mathcal U_0,
\end{equation}
and
\begin{equation}\label{eq:polar-bernstein}
 |\nabla\rho|^2+\rho^2|k|^2
 =|\nabla u|^2=e\leq h
 \qquad\text{on }\mathcal U_0.
\end{equation}
In particular,
\begin{equation}\label{eq:k-zero}
 \lim_{R\to\infty}
 \sup_{\substack{x\in\Omega\\ |x|\geq R}}|k(x)|=0.
\end{equation}
\end{lemma}
\begin{proof}
By \eqref{eq:rho-n-def}, and
\eqref{eq:n-derivative-identities} we have 
\eqref{eq:polar-first-derivative}.  Differentiating once more and
summing over $a=1,2$, we obtain
\[
\begin{aligned}
 \Delta u
 &=\sum_{a=1}^2\partial_a
   \bigl((\partial_a\rho)n+\rho k_a n^\perp\bigr)\\
 &=\bigl(\Delta\rho-\rho|k|^2\bigr)n
   +\bigl(2\nabla\rho\cdot k+\rho\Div k\bigr)n^\perp.
\end{aligned}
\]
Since $u=\rho n$ and $\Delta u=-hu=-\rho h n$,
comparison of the coefficients in the orthonormal basis
$(n,n^\perp)$, we obtain 
\[
 -\Delta\rho+\rho|k|^2=\rho h,
 \qquad
 2\nabla\rho\cdot k+\rho\Div k=0.
\]
Multiplying the second identity by $\rho$ we get 
$ \Div(\rho^2k)=0$, which  proves \eqref{eq:amplitude-phase}.

Taking squared norms in \eqref{eq:polar-first-derivative} and using
$n\cdot n^\perp=0$ we infer 
\[
 |\partial_a u|^2=(\partial_a\rho)^2+\rho^2k_a^2.
\]
Summation over $a$ and Lemma~\ref{lem:basic-bounds} we prove
\eqref{eq:polar-bernstein}.  By \eqref{eq:rho-lower},
\[
 |k|^2\leq\rho^{-2}h\leq4h
 \qquad\text{on }\mathcal U_0.
\]
Since $h(x)=1-|u(x)|^2\to0$ uniformly as $|x|\to\infty$,
\eqref{eq:k-zero} follows.
\end{proof}

\subsection{The Jacobi quadratic form}
\label{subsec:jacobi-form}

Since $h(x)\to0$ and $k(x)\to0$ as $|x|\to\infty$, we can choose
$\widehat R_0\geq R_0+1$ such that
\[
 3h(x)+7|k(x)|^2\leq1
 \quad\forall\,x\in\mathbb{R}^2
 \text{ with }|x|>\widehat R_0.
\]
Replace $R_0$ by $\widehat R_0$ in
\eqref{eq:exterior-domains}, restrict $\rho$, $n$, and $k$ to the new
domains, and keep the notation $\mathcal U_0$ and $\Omega$.  Then
\begin{equation}\label{eq:small-coercive}
 3h+7|k|^2\leq1
 \qquad\text{on }\Omega.
\end{equation}
Equations \eqref{eq:rho-lower}, \eqref{eq:k-curl},
\eqref{eq:amplitude-phase}, \eqref{eq:polar-bernstein}, and
\eqref{eq:k-zero} remain valid on the restricted domains.

For $\zeta\in C^\infty(\Omega;\mathbb{R}^2)$, set
\begin{equation*}\label{eq:Jacobi-operator-def}
 L\zeta:=-\Delta\zeta-h\zeta+2(u\cdot\zeta)u.
\end{equation*}
For $\zeta\in C_c^\infty(\Omega;\mathbb{R}^2)$, define the associated quadratic
form
\begin{equation}\label{eq:Q-def}
 Q_\Omega(\zeta)
 :=\int_\Omega
 \bigl(|\nabla\zeta|^2-h|\zeta|^2+2(u\cdot\zeta)^2\bigr)\dd x.
\end{equation}
Integration by parts gives
\begin{equation*}\label{eq:Q-L-relation}
 Q_\Omega(\zeta)=\int_\Omega\zeta\cdot L\zeta\,\dd x.
\end{equation*}
Every $\zeta\in C_c^\infty(\Omega;\mathbb{R}^2)$ has the unique decomposition
\begin{equation}\label{eq:zeta-decomp}
 \zeta=\alpha n+\rho\beta n^\perp,
 \qquad
 \alpha:=n\cdot\zeta,
 \qquad
 \beta:=\rho^{-1}n^\perp\cdot\zeta,
\end{equation}
where $\alpha,\beta\in C_c^\infty(\Omega)$.

\begin{lemma}
\label{lem:coercivity}
Let $\zeta\in C_c^\infty(\Omega;\mathbb{R}^2)$ and let $\alpha,\beta$ be the
unique functions in \eqref{eq:zeta-decomp}.  Then
\begin{equation}\label{eq:Q-coercive}
 Q_\Omega(\zeta)
 \geq\frac12\int_\Omega
 \bigl(|\nabla\alpha|^2+\alpha^2
       +\rho^2|\nabla\beta|^2\bigr)\dd x.
\end{equation}
\end{lemma}

\begin{proof}
We first claim that
\begin{equation}\label{eq:Q-exact}
\begin{aligned}
 Q_\Omega(\zeta)
 =\int_\Omega\bigl(&
 |\nabla\alpha|^2+(|k|^2+2-3h)\alpha^2
 +\rho^2|\nabla\beta|^2
 +4\alpha\rho k\cdot\nabla\beta
 \bigr)\dd x.
\end{aligned}
\end{equation}
For each $a\in\{1,2\}$, by 
\eqref{eq:n-derivative-identities} we have 
\begin{equation}\label{eq:partial-zeta}
 \partial_a\zeta
 =\bigl(\partial_a\alpha-\rho\beta k_a\bigr)n
  +\bigl(\alpha k_a+\beta\partial_a\rho
          +\rho\partial_a\beta\bigr)n^\perp.
\end{equation}
Since $(n,n^\perp)$ is orthonormal, using 
\eqref{eq:partial-zeta} we get 
\begin{equation}\label{eq:gradient-zeta}
 |\nabla\zeta|^2
 =|\nabla\alpha-\rho\beta k|^2
  +|\alpha k+\beta\nabla\rho+\rho\nabla\beta|^2.
\end{equation}
Also,
\begin{equation}\label{eq:zeta-potential-components}
 |\zeta|^2=\alpha^2+\rho^2\beta^2,
 \qquad
 u\cdot\zeta=\rho\alpha.
\end{equation}

To treat the terms quadratic in $\beta$, we rewrite the first equation in
\eqref{eq:amplitude-phase} as
\begin{equation}\label{eq:rho-linear-equation}
 -\Delta\rho+\rho(|k|^2-h)=0.
\end{equation}
Multiplying \eqref{eq:rho-linear-equation} by
$\rho\beta^2\in C_c^\infty(\Omega)$ and
integrating by parts we have 
\[\begin{aligned}
 0 &=\int_\Omega
   \bigl[\nabla\rho\cdot\nabla(\rho\beta^2)
          +\rho^2\beta^2(|k|^2-h)\bigr]\dd x \\
 &=\int_\Omega
   \bigl[|\nabla(\rho\beta)|^2
          -\rho^2|\nabla\beta|^2
          +\rho^2\beta^2(|k|^2-h)\bigr]\dd x.
\end{aligned}\]
%In the second equality we used the pointwise identity
%\[
% \nabla\rho\cdot\nabla(\rho\beta^2)
% =|\nabla(\rho\beta)|^2-\rho^2|\nabla\beta|^2.\]
Thus
\begin{equation}\label{eq:beta-cancellation}
 \int_\Omega
 \bigl(|\nabla(\rho\beta)|^2
       +\rho^2\beta^2|k|^2-h\rho^2\beta^2\bigr)\dd x
 =\int_\Omega\rho^2|\nabla\beta|^2\dd x.
\end{equation}

The mixed terms in \eqref{eq:gradient-zeta} have integral
\begin{equation}\label{eq:mixed-before-ibp}
 I_{\mathrm{mix}}
 =2\int_\Omega
 \bigl(-\rho\beta k\cdot\nabla\alpha
       +\alpha\beta k\cdot\nabla\rho
       +\alpha\rho k\cdot\nabla\beta\bigr)\dd x.
\end{equation}
The second equation in \eqref{eq:amplitude-phase} and $\rho>0$ imply
\begin{equation}\label{eq:div-k-rho}
 \rho\Div k=-2k\cdot\nabla\rho.
\end{equation}
Since $\alpha$ and $\beta$ are compactly supported in $\Omega$, 
integration by parts we get 
\begin{align*}
 -2\int_\Omega\rho\beta k\cdot\nabla\alpha\,\dd x
 &=2\int_\Omega\alpha\Div(\rho\beta k)\,\dd x\\
 &=2\int_\Omega\alpha
   \bigl(\beta k\cdot\nabla\rho
         +\rho k\cdot\nabla\beta
         +\rho\beta\Div k\bigr)\dd x\\
 &=2\int_\Omega\alpha
   \bigl(-\beta k\cdot\nabla\rho
         +\rho k\cdot\nabla\beta\bigr)\dd x,
\end{align*}
where the last equality follows from \eqref{eq:div-k-rho}.  Substitution
into \eqref{eq:mixed-before-ibp} yields
\begin{equation}\label{eq:mixed-cancellation}
 I_{\mathrm{mix}}
 =4\int_\Omega\alpha\rho k\cdot\nabla\beta\,\dd x.
\end{equation}

The terms depending only on $\alpha$ in $Q_\Omega(\zeta)$ are
\[
 \int_\Omega
 \bigl(|\nabla\alpha|^2+(|k|^2-h+2\rho^2)\alpha^2\bigr)\dd x.
\]
Since $\rho^2=|u|^2=1-h$, we have
\begin{equation}\label{eq:alpha-coefficient}
 |k|^2-h+2\rho^2=|k|^2+2-3h.
\end{equation}
Combining \eqref{eq:gradient-zeta},
\eqref{eq:zeta-potential-components}, \eqref{eq:beta-cancellation},
\eqref{eq:mixed-cancellation}, and \eqref{eq:alpha-coefficient}, we prove 
\eqref{eq:Q-exact}.

For the coercive estimate, by Young's inequality we get 
\begin{equation}\label{eq:mixed-young}
 4|\alpha|\rho|k||\nabla\beta|
 \leq\frac12\rho^2|\nabla\beta|^2+8|k|^2\alpha^2.
\end{equation}
Using \eqref{eq:mixed-young} in \eqref{eq:Q-exact}, we obtain
\begin{align*}
 Q_\Omega(\zeta)
 &\geq\int_\Omega
 \left(|\nabla\alpha|^2
 +(2-3h-7|k|^2)\alpha^2
 +\frac12\rho^2|\nabla\beta|^2\right)\dd x.
\end{align*}
By \eqref{eq:small-coercive} we infer
$2-3h-7|k|^2\geq1$.  Hence
\[
 Q_\Omega(\zeta)
 \geq\int_\Omega
 \left(|\nabla\alpha|^2+\alpha^2
       +\frac12\rho^2|\nabla\beta|^2\right)\dd x,
\]
which implies \eqref{eq:Q-coercive}.
\end{proof}

\subsection{Square-integrability estimates for amplitude derivatives}
\label{subsec:amplitude-L2}

The amplitude itself does not belong to $L^2(\Omega)$ in general,
since $\rho(x)\to1$.  We instead estimate its first and second
derivatives.  %The latter estimate also gives $-\Delta\rho/\rho\in L^2(\Omega)$.

\begin{proposition}
\label{prop:transverse-energy}
Let $\Omega$ be the exterior domain fixed after
\eqref{eq:small-coercive}.  Then
\begin{equation}\label{eq:transverse-energy}
 \int_\Omega
 \bigl(|\nabla\rho|^2+|D^2\rho|^2+|\nabla k|^2\bigr)\dd x<\infty.
\end{equation}
\end{proposition}

\begin{proof}
For $j\in\{1,2\}$, set $w_j:=\partial_j u$.  Differentiating
$-\Delta u=hu$ and using
$\partial_j h=-2u\cdot\partial_j u=-2u\cdot w_j$ gives
\begin{equation}\label{eq:Jacobi}
 Lw_j=-\Delta w_j-hw_j+2(u\cdot w_j)u=0
 \qquad\text{pointwise in }\Omega.
\end{equation}
Let $\chi\in C_c^\infty(\Omega)$.  Multiplying
\eqref{eq:Jacobi} by $\chi^2w_j$, integrating over $\Omega$ and using integration by parts, we have
\begin{equation}\label{eq:IMS}
 Q_\Omega(\chi w_j)
 =\int_\Omega|\nabla\chi|^2|w_j|^2\dd x.
\end{equation}

By \eqref{eq:polar-first-derivative}, we have
\begin{equation}\label{eq:wj-decomp}
 w_j=(\partial_j\rho)n+\rho k_j n^\perp.
\end{equation}
Therefore, in the decomposition \eqref{eq:zeta-decomp} of
$\chi w_j$, one has
\begin{equation}\label{eq:chi-wj-components}
 \alpha=\chi\,\partial_j\rho,
 \qquad
 \beta=\chi k_j.
\end{equation}

Choose $\eta\in C^\infty(\Omega;[0,1])$ such that
\begin{equation*}\label{eq:inner-cutoff}
\begin{gathered}
 \eta(x)=0\quad\text{if }R_0<|x|\leq R_0+\frac12,
 \qquad
 \eta(x)=1\quad\text{if }|x|\geq R_0+1,\\
 \supp\nabla\eta\Subset\Omega.
\end{gathered}
\end{equation*}
Choose $\xi\in C^\infty([0,\infty);[0,1])$ satisfying
\begin{equation*}\label{eq:outer-cutoff-profile}
 \xi(t)=1\quad\forall\,t\in[0,1],
 \qquad
 \xi(t)=0\quad\forall\,t\in[2,\infty),
\end{equation*}
and define, for $R>0$,
\begin{equation*}\label{eq:cutoff-def}
 \xi_R(x):=\xi\left(\frac{|x|}{R}\right),
 \qquad
 \chi_R:=\eta\xi_R.
\end{equation*}
Let $R_\ast>2(R_0+1)$.  For $R\geq R_\ast$,
$\chi_R\in C_c^\infty(\Omega)$, the supports of $\nabla\eta$ and
$\nabla\xi_R$ are disjoint, and
\[
 |\nabla\xi_R|\leq\frac{C}{R},
 \qquad
 \supp\nabla\xi_R\subset\overline{\Acal(R,2R)}.
\]
Moreover, $\xi_R=1$ on $\supp\nabla\eta$ and
$\eta=1$ on $\supp\nabla\xi_R$.  It follows that
\begin{align}
 \int_\Omega|\nabla\chi_R|^2e\,\dd x
 &\leq
 \underbrace{\int_{\supp\nabla\eta}|\nabla\eta|^2e\,\dd x}_{=:C_{\mathrm{in}}}
 +\frac{C}{R^2}\int_{\Acal(R,2R)}e\,\dd x \notag\\
 &\leq C_{\mathrm{in}}
 +\frac{C}{R^2}\mathcal L^2(\Acal(R,2R))
      \sup_{\Acal(R,2R)}h \notag\\
 &\leq C_{\mathrm{in}}+C\sup_{\Acal(R,2R)}h
 \leq C_0,
 \label{eq:cutoff-error}
\end{align}
where we used \eqref{eq:h-e-bounds},
$\mathcal L^2(\Acal(R,2R))=3\pi R^2$, and
$\lim_{|x|\to\infty}h(x)=0$.  The positive constants $C$ and $C_0$ are
independent of $R\geq R_\ast$, and $C_{\mathrm{in}}<\infty$ because
$\supp\nabla\eta\Subset\Omega$ and $e\in C^\infty(\Omega)$.

We apply \eqref{eq:Q-coercive} and \eqref{eq:IMS} to $\chi_R w_j$.
In view of \eqref{eq:chi-wj-components}, $|w_j|^2\leq e$, and \eqref{eq:cutoff-error}, we obtain 
\begin{equation}\label{eq:coercive-cutoff-bound}
 \int_\Omega\bigl(
 |\nabla(\chi_R\partial_j\rho)|^2
 +\chi_R^2|\partial_j\rho|^2
 +\rho^2|\nabla(\chi_R k_j)|^2\bigr)\dd x
 \leq C_1,
\end{equation}
where $C_1$ is independent of $R\geq R_\ast$.

By the product rule and the inequality
$|a-b|^2\leq2|a|^2+2|b|^2$, we have
\begin{align}
 \chi_R^2|\nabla\partial_j\rho|^2
 &\leq2|\nabla(\chi_R\partial_j\rho)|^2
      +2|\partial_j\rho|^2|\nabla\chi_R|^2,
 \label{eq:rho-cutoff-expand}\\
 \rho^2\chi_R^2|\nabla k_j|^2
 &\leq2\rho^2|\nabla(\chi_R k_j)|^2
      +2\rho^2k_j^2|\nabla\chi_R|^2.
 \label{eq:k-cutoff-expand}
\end{align}
Furthermore, \eqref{eq:wj-decomp} and orthogonality imply
\begin{equation}\label{eq:wj-energy-component}
 |\partial_j\rho|^2+\rho^2k_j^2=|w_j|^2\leq e.
\end{equation}
Combining \eqref{eq:cutoff-error},
\eqref{eq:coercive-cutoff-bound},
\eqref{eq:rho-cutoff-expand}, \eqref{eq:k-cutoff-expand}, and
\eqref{eq:wj-energy-component}, we obtain
\begin{equation}\label{eq:fatou-uniform}
 \sup_{R\geq R_\ast}\int_\Omega\chi_R^2
 \bigl(|\partial_j\rho|^2+|\nabla\partial_j\rho|^2
       +\rho^2|\nabla k_j|^2\bigr)\dd x<\infty.
\end{equation}

Choose any sequence $R_m\to\infty$ with $R_m\geq R_\ast$.  For each
fixed $x\in\Omega$, one has $\xi_{R_m}(x)\to1$ and hence
$\chi_{R_m}(x)\to\eta(x)$.  Fatou's lemma applied to the nonnegative
integrands in \eqref{eq:fatou-uniform} yields
\begin{equation}\label{eq:eta-energy-bound}
 \int_\Omega\eta^2
 \bigl(|\partial_j\rho|^2+|\nabla\partial_j\rho|^2
       +\rho^2|\nabla k_j|^2\bigr)\dd x<\infty.
\end{equation}
On $\{x\in\mathbb{R}^2:|x|\geq R_0+1\}$, one has $\eta=1$ and
$\rho\geq\frac12$.  The annulus
$\{x\in\mathbb{R}^2:R_0<|x|<R_0+1\}$ has compact closure in $\mathcal U_0$, where
$\rho$ and $k$ are smooth.  Therefore \eqref{eq:eta-energy-bound} implies
\[
 \partial_j\rho\in L^2(\Omega),
 \qquad
 \nabla\partial_j\rho\in L^2(\Omega;\mathbb{R}^2),
 \qquad
 \nabla k_j\in L^2(\Omega;\mathbb{R}^2).
\]
Summing over $j=1,2$ gives \eqref{eq:transverse-energy}.
\end{proof}

\begin{corollary}
\label{cor:forced-phase}
Define
\begin{equation}\label{eq:sigma-def}
 \sigma:=h-|k|^2
\end{equation}
and set $F:=\sigma k$.  Then
\begin{equation}\label{eq:forced-phase-corollary-conclusion}
 \sigma\in L^2(\Omega),
 \quad
 F\in L^2(\Omega;\mathbb{R}^2),
 \quad
 \Div\bigl((1-|k|^2)k-F\bigr)=0
 \quad\text{in }\mathcal D'(\Omega).
\end{equation}
\end{corollary}

\begin{proof}
Since $h$, $k$, and $\rho$ are smooth on $\mathcal U_0$ and
$\rho$ is strictly positive,
$\sigma\in C^\infty(\mathcal U_0)$.  Rearranging the first identity in
\eqref{eq:amplitude-phase}, we obtain the pointwise identity
\begin{equation}\label{eq:sigma-amplitude}
 \sigma=-\frac{\Delta\rho}{\rho}\qquad\text{in }\Omega.
\end{equation}
Since $|\Delta\rho|\leq\sqrt2\,|D^2\rho|$,
Proposition~\ref{prop:transverse-energy},
\eqref{eq:sigma-amplitude}, and $\rho\geq\frac12$ give
\[
 \|\sigma\|_{L^2(\Omega)}
 \leq2\sqrt2\,\|D^2\rho\|_{L^2(\Omega)}<\infty.
\]
Moreover, \eqref{eq:polar-bernstein}, \eqref{eq:rho-lower}, and
\eqref{eq:h-e-bounds} imply
\[
 |k|^2\leq\rho^{-2}h\leq4
 \qquad\text{on }\Omega.
\]
Consequently, we have
\[
 \|F\|_{L^2(\Omega;\mathbb R^2)}
 \leq\|k\|_{L^\infty(\Omega)}
      \|\sigma\|_{L^2(\Omega)}<\infty.
\]
It follows from \eqref{eq:sigma-def} and
$\rho^2=|u|^2=1-h$ that $\rho^2=1-|k|^2-\sigma$. By the definition of $A$ in
\eqref{eq:autonomous-flux-potential} and the identity $F=\sigma k$,
we have 
\[
 A(k)-F
 =\bigl(1-|k|^2-\sigma\bigr)k
 =\rho^2k
 \qquad\text{on }\Omega.
\]
Hence, the second identity in \eqref{eq:amplitude-phase} gives
$\Div(A(k)-F)=0$ in $\mathcal D'(\Omega)$.  As a result,
\[
 \int_\Omega A(k)\cdot\nabla\varphi\,\dd x
 =\int_\Omega F\cdot\nabla\varphi\,\dd x
 \quad\forall\,\varphi\in C_c^\infty(\Omega).
\]
\end{proof}

\subsection{Proof of Theorem~\ref{thm:main}}
\label{subsec:proof-open-problem}

\begin{proof}[Proof of Theorem~\ref{thm:main}]
Let $\mathcal U_0$ and $\Omega$ be the exterior domains given in \eqref{eq:exterior-domains}.  The field $k$ is smooth on the open
neighborhood $\mathcal U_0$ of $\overline\Omega$. Furthermore, \eqref{eq:k-curl} gives $\curl k=0$, \eqref{eq:k-zero} gives the required uniform decay, and Corollary~\ref{cor:forced-phase} gives $F\in L^2(\Omega;\mathbb R^2)$ and $\Div(A(k)-F)=0$ in $\mathcal D'(\Omega)$.  Thus, every hypothesis of Theorem~\ref{thm:exterior-phase} is satisfied. According to Theorem~\ref{thm:exterior-phase}, we obtain  $k\in L^4(\Omega;\mathbb R^2)$ and  $|k|^2\in L^2(\Omega)$.  Combining  \eqref{eq:sigma-def} with
Corollary~\ref{cor:forced-phase}, we get
\[
 h=|k|^2+\sigma\in L^2(\Omega).
\]
Since $u\in C^\infty(\mathbb{R}^2;\mathbb{R}^2)$, we have
\[
 \int_{B_{R_0}(0)}h^2\dd x<\infty.
\]
Combining this with $h\in L^2(\Omega)$, we conclude that
$h\in L^2(\mathbb{R}^2)$. Hence,
\[
 \int_{\mathbb{R}^2}(1-|u|^2)^2\dd x
 =\int_{\mathbb{R}^2}h^2\dd x<\infty.
\]
\end{proof}

\begin{remark}
For every sufficiently large $R$, the map $u$ has no zeros on
$\partial B_R(0)$, and
\[
 \deg(u,\infty):=
 \deg\!\left(
 \left.\frac{u}{|u|}\right|_{\partial B_R(0)}
 \right)
 =\frac1{2\pi}
 \int_{\partial B_R(0)}k\cdot\tau\,\dd\mathcal H^1
\]
is independent of $R$.  Combining Theorem~\ref{thm:main} with the
quantization theorem of Brezis--Merle--Rivi\`ere
\cite[Theorem~1 and Remark~1.1]{BMR} gives
\[
 \int_{\mathbb R^2}(1-|u|^2)^2\dd x
 =2\pi\deg(u,\infty)^2.
\]
\end{remark}

\bigskip

\section*{\bf Acknowledgements}
Hong-Ge Chen was supported by the National Natural Science Foundation of China (Grant Nos.~12201607 and~12571249) and the Postdoctoral Project of Hubei Province (Grant No.~2024HBBHXF095). Juncheng Wei was supported by National R\&D Program of China (Grant No. 2022YFA1005602), and Hong Kong General Research Fund ``New frontiers in singular limits of nonlinear partial differential equation". Wen Yang was supported by the National Key Research and Development Program of China (Grant No.~2022YFA1006800), the National Natural Science Foundation of China
(Grant Nos. 12271369 and~12531010), the Science and Technology Development Fund of the Macao SAR (FDCT, Grant No.~0070/2024/RIA1), the
Start-up Research Grant of the University of Macau (Grant No.~SRG2023-00067-FST), the Multi-Year Research Grants of the University
of Macau (Grant Nos.~MYRG-GRG2024-00082-FST-UMDF and MYRG-GRG2025-00051-FST), and the University of Macau Development
Foundation (Grant No.~TISF/2025/006/FST). 
\medskip

\noindent {\bf Data availability statement}: There are no new data associated with this article.
\medskip

\noindent {\bf AI assistance statement}: The authors used OpenAI models to assist with language polishing and manuscript editing.

\end{document}